\documentclass{amsart}
\usepackage{amsfonts,amsmath,amssymb}
\usepackage{mathrsfs,mathtools}
\usepackage{enumerate}
\usepackage{xspace}
\usepackage{hyperref}
\usepackage{esint}
\usepackage{graphicx}
\DeclareMathAlphabet{\mathpzc}{OT1}{pzc}{m}{it}

\usepackage{color}

\newcommand{\sr}{{\tfrac12}}

\newcommand{\R}{\mathbb{R}}
\newcommand{\F}{\mathcal{F}}
\newcommand{\Y}{\mathpzc{Y}}

\newcommand{\Fm}{\mathpzc{F}}

\newcommand{\N}{\mathpzc{N}}

\newcommand{\vero}{\texttt{v}}

\newcommand{\C}{\mathcal{C}}

\newcommand{\V}{\mathbb{V}}
\newcommand{\W}{\mathbb{W}}
\newcommand{\T}{\mathscr{T}}

\newcommand{\Q}{\mathbb{Q}}

\newcommand{\Ss}{\mathscr{S}}
\newcommand{\Tm}{\mathcal{T}}

\newcommand{\HL}{ \mbox{ \raisebox{7.4pt} {\tiny$\circ$} \kern-10.3pt} {H_L^1} }
\newcommand{\Wp}{ \mbox{ \raisebox{7.7pt} {\scriptsize$\circ$} \kern-10.1pt} {W^1_p} }
\newcommand{\Wpp}{ \mbox{ \raisebox{7.7pt} {\scriptsize$\circ$} \kern-10.1pt} {W^1_{p'}} }
\newcommand{\Sz}{ \mbox{ \raisebox{7.5pt} {\scriptsize$\circ$} \kern-10.1pt} {\Ss} }
\newcommand{\HLnew}{ \mbox{ \raisebox{7pt} {\scriptsize$\circ$} \kern-10.1pt}{H}^1_L }
\newcommand{\HLn}{{\mbox{\,\raisebox{5.1pt} {\tiny$\circ$} \kern-9.1pt}{H}^1_L  }}
\newcommand{\HLs}{{\mbox{\raisebox{8.7pt} {\scriptsize$\circ$} \kern-10.1pt}{H}^1_L  }}
\newcommand{\ue}{\mathscr{U}}
\newcommand{\verot}{\emph{\texttt{v}}}
\newcommand{\Tr}{\mathbb{T}}

\newcommand{\bu}{\textbf{u}}
\newcommand{\poverp}{p'\!/p}

\DeclareMathOperator*{\tr}{tr_\Omega}

\DeclareMathOperator*{\signum}{sign}
\DeclareMathOperator*{\supp}{supp}
\DeclareMathOperator*{\diag}{diag}

\newcommand{\GRAD}{\nabla}
\newcommand{\DIV}{\textrm{div}}
\newcommand{\diff}{\, \mbox{\rm d}}
\newcommand{\ie}{i.e.,\@\xspace}
\newcommand{\cf}{cf.\@\xspace}
\newcommand{\Hs}{\mathbb{H}^s(\Omega)}
\newcommand{\Ws}{\mathbb{H}^{1-s}(\Omega)}
\DeclareMathOperator*{\diam}{diam}

\newcommand{\A}{\mathcal{A}}
\newcommand{\Loneloc}{{L^1_{\mathrm{loc}}}}
\newcommand{\Hsd}{\mathbb{H}^{-s}(\Omega)}

\newtheorem{theorem}{Theorem}[section]
\newtheorem{lemma}[theorem]{Lemma}
\newtheorem{proposition}{Proposition}[section]
\newtheorem{corollary}[theorem]{Corollary}
\theoremstyle{definition}
\newtheorem{definition}[theorem]{Definition}
\theoremstyle{remark}
\newtheorem{remark}[theorem]{Remark}

\numberwithin{equation}{section}

\DeclareMathOperator{\Sob}{sob}
\newcommand{\dist}{{\textup{\textsf{d}}_{x_0}}}
\newcommand{\distb}{{\textup{\textsf{d}}_{x_0}}^{\kern-0.8em\beta}}
\newcommand{\calL}{{\mathcal L}}
\newcommand{\Wmpo}{ \mbox{ \raisebox{7.4pt} {\tiny$\circ$} \kern-10.7pt} {W_p^m} }
\newcommand{\Wonepo}{ \mbox{ \raisebox{7.4pt} {\tiny$\circ$} \kern-10.7pt} {W_p^1} }
\newcommand{\Nin}{\,{\mbox{\,\raisebox{6.2pt} {\tiny$\circ$} \kern-11.1pt}\N }}
\newcommand{\Ninn}{{\mbox{\,\raisebox{4.5pt} {\tiny$\circ$} \kern-8.8pt}\N }}
\newcommand{\polX}{{\mathbb X}}
\newcommand{\polY}{{\mathbb Y}}
\begin{document}
\title[Polynomial interpolation in weighted spaces]{Piecewise polynomial interpolation in Muckenhoupt weighted 
Sobolev spaces and applications}

\author[R.H.~Nochetto]{Ricardo H.~Nochetto}
\address{Department of Mathematics and Institute for Physical Science and Technology,
University of Maryland, College Park, MD 20742, USA.}
\email{rhn@math.umd.edu}
\thanks{RHN has been partially supported by NSF grants DMS-1109325 and DMS-1411808.}

\author[E.~Ot\'arola]{Enrique Ot\'arola}
\address{Department of Mathematics, University of Maryland, College Park, MD 20742, USA and
Department of Mathematical Sciences, George Mason University, Fairfax, VA 22030, USA.}
\email{kike@math.umd.edu}
\thanks{EO has been partially supported by the Conicyt-Fulbright Fellowship Beca Igualdad de Oportunidades and NSF grants DMS-1109325 and DMS-1411808.}

\author[A.J.~Salgado]{Abner J.~Salgado}
\address{Department of Mathematics, University of Tennessee, Knoxville, TN 37996, USA.}
\email{asalgad1@utk.edu}
\thanks{AJS has been partially supported by NSF grant DMS-1418784.}


\date{Version of \today.}


\begin{abstract}
We develop a constructive piecewise polynomial approximation theory in
weighted Sobolev spaces with Muckenhoupt weights for any polynomial degree.
The main ingredients to derive optimal error estimates for an
averaged Taylor polynomial are a suitable weighted Poincar\'e
inequality, a cancellation property and a simple induction argument.
We also construct a quasi-interpolation operator, built on local
averages over stars, which is well defined for functions in $L^1$.
We derive optimal error estimates for any polynomial degree on
simplicial shape regular meshes.
On rectangular meshes, these estimates are valid under the condition
that neighboring elements have comparable size, which yields
optimal anisotropic error estimates over $n$-rectangular domains.
The interpolation theory extends to cases when the error and function
regularity require different weights.
We conclude with three applications: nonuniform elliptic boundary value problems, elliptic problems with singular sources, and fractional 
powers of elliptic operators.
\end{abstract}

\maketitle

\section{Introduction}
\label{sec:introduccion}

A fundamental tool in analysis, with both practical and theoretical relevance,
is the approximation of a function by a simpler one.
For continuous functions a foundational result in this direction was given by K.~Weierstrass 
in 1885: continuous functions defined on a compact interval can be uniformly approximated as 
closely as desired by polynomials.
Mollifiers, interpolants, splines and even Nevanlinna-Pick theory can be regarded as instances of 
this program; see, for instance, \cite{AMBook,LS:07}.
For weakly differentiable functions, the approximation by polynomials is very useful when trying to understand 
their behavior. In fact, this idea goes back to S.L.~Sobolev \cite{Sobolev}, 
who used a sort of averaged Taylor polynomial to discuss equivalent norms in Sobolev spaces.

The role of polynomial approximation and error estimation
is crucial in numerical analysis: it is the basis of discretization techniques 
for partial differential equations (PDE), particularly the finite element method. For the latter, 
several constructions for standard Sobolev spaces $W_p^{1}$, with 
$1 \leq p \leq \infty$, and their properties are well studied; see 
\cite{Clement,DupScott,Duran83,DL:05,SZ:90}.

On the other hand, many applications lead to boundary value problems for nonuniformly 
elliptic equations. The ellipticity distortion can be caused by degenerate/singular 
behavior of the coefficients of the differential operator or by singularities in the domain.
For such equations
it is natural to look for solutions in weighted Sobolev spaces
\cite{AGM,BBD:06,CS:11,CS:07,DAngelo:SINUM,DL:10,FKS:82,FS:12,Kufner,Turesson}
and to study the regularity properties of the solution in weighted spaces
as well \cite{KufnerSandig}. Of particular 
importance are weighted Sobolev spaces with a weight belonging
to the so-called Muckenhoupt class $A_p$ 
\cite{Muckenhoupt}; see also \cite{FKS:82,Haroske,Turesson}. However, 
the literature focusing on polynomial approximation in this type of Sobolev 
spaces is rather scarce; we refer the reader to
\cite{AGM,Apel:99,ABH:06,BBD:06,DAngelo:SINUM,French,GP:06,LIH} 
for some partial results.
Most of these results focus on a particular nonuniformly elliptic equation and exploit
the special structure of the coefficient to derive polynomial interpolation results.

To fix ideas, consider the following nonuniformly elliptic boundary value problem: let $\Omega$ be an 
open and bounded subset of $\R^n$ ($n\ge1$) with boundary $\partial\Omega$. Given a function $f$,
find $u$ that solves
\begin{equation}
\label{weighted_second}
  \begin{dcases}
    -\DIV (\A(x) \nabla u) = f, & \text{in } \Omega, \\
    u = 0, & \text{on } \partial\Omega,
  \end{dcases}
\end{equation}
where 
$\A : \Omega \rightarrow \R^{n\times n}$ is symmetric and satisfies the following nonuniform ellipticity condition
\begin{equation}
 \omega(x) |\xi|^2 \lesssim \xi^\intercal \A(x) \xi \lesssim \omega(x) |\xi|^2,
 \quad \forall \xi \in \R^n, \quad a.e.~\Omega.
\label{eq:weightissingular}
\end{equation}
Here the relation $a \lesssim b$ indicates that $a \leq Cb$, with a constant $C$ and $\omega$ is a weight 
function, \ie a nonnegative and locally integrable measurable function, which might 
vanish, blow up, and possess singularities. 
Examples of this type of equations are the harmonic extension problem related with the 
fractional Laplace operator \cite{CS:11,CS:07,NOS}, elliptic problems involving 
 measures 
\cite{AGM,DAngelo:SINUM}, elliptic PDE in an axisymmetric three dimensional domain with axisymmetric data 
\cite{BBD:06,GP:06}, and equations modeling the motion of particles in a central potential field in quantum 
mechanics \cite{ABH:06}. Due to the nature of the coefficient $\A$, the classical Sobolev 
space $H^1(\Omega)$ is not appropriate for the analysis and approximation of this problem.

Nonuniformly elliptic equations of the type \eqref{weighted_second}--\eqref{eq:weightissingular},
with $\omega$ in the so-called Muckenhoupt class $A_2$, have been studied in \cite{FKS:82}:
for $f \in L^2(\omega^{-1},\Omega)$, there exists a unique 
solution in $H_0^1(\omega,\Omega)$ \cite[Theorem 2.2]{FKS:82} (see \S~\ref{sub:sec:weighted_spaces} for notation). 
Consider the discretization of \eqref{weighted_second} with the finite element method.
Let $\T$ be a conforming triangulation of $\Omega$ and let $\V(\T)$ be a finite element space. The Galerkin 
approximation of the solution to \eqref{weighted_second} is given by the unique function $U_{\T} \in \V(\T)$ 
that solves
\begin{equation}
\label{weighted_second_discrete}
  \int_{\Omega} \A \nabla U_{\T} \cdot \nabla W = \int_{\Omega} f W,
  \quad \forall W \in \V(\T).
\end{equation}
Invoking Galerkin orthogonality, we deduce
\begin{equation}
\label{GO}
   \| u - U_{\T} \|_{H_0^1(\omega,\Omega)} \lesssim
      \inf_{W \in \V(\T_{\Y})} \| u - W \|_{H_0^1(\omega,\Omega)}.
\end{equation}
In other words, the numerical analysis of this boundary value problem reduces to a result in approximation 
theory: the distance between the exact solution $u$ and its approximation $U_{\T}$ in a finite element space 
is bounded by 
the best approximation error in the finite element space with respect to an appropriate weighted Sobolev 
norm. A standard way of obtaining bounds for the approximation error is by considering $W = \Pi_{\T} v$ in 
\eqref{GO}, where $\Pi_\T$ is a suitable interpolation operator. 

The purpose of this work is twofold. We first go back to the basics, 
and develop an elementary
constructive approach to piecewise polynomial interpolation in
weighted Sobolev spaces with Muckenhoupt weights. We consider an
averaged version of the Taylor polynomial and, upon using an 
appropriate weighted Poincar\'e inequality and a cancellation
property, we derive optimal approximation
estimates for constant and linear approximations. We extend these results
to any polynomial degree $m$ ($m \geq 0$), by a simple induction argument. 

The functional framework considered is 
weighted Sobolev spaces with weights in the Muckenhoupt class
$A_p(\R^n)$, thereby extending the classical polynomial approximation theory in Sobolev spaces
\cite{BrennerScott,Ciarletbook,Clement,SZ:90}. 
In addition, we point out
that the results about interpolation in Orlicz spaces of \cite{DR:07,Duran:87} do not apply to our 
situation since, for weighted spaces, the Young function used to define the modular depends on the point in space as well.
In this respect, our results can be regarded as a first step in the development of an approximation theory in Orlicz-Musielak spaces
and in Sobolev spaces in metric measure spaces \cite{H:96}.

The second main contribution of this work is the 
construction of a quasi-inter\-polation operator
$\Pi_{\T}$, built on local averages over stars and thus well defined for functions in $L^1(\Omega)$
as those in \cite{Clement,SZ:90}.
The ensuing polynomial approximation theory in weighted Sobolev spaces with Muckenhoupt weights
allows us to obtain optimal and local interpolation estimates for the 
quasi-interpolant $\Pi_{\T}$. 
On simplicial discretizations, these results hold true for any polynomial degree $m \geq 0$,
and they are derived in the weighted $W_p^k$-seminorm ($0 \leq k \leq m+1$). 
The key ingredient is an invariance property of 
the quasi-interpolant $\Pi_{\T}$ over the finite element space. 
On the other hand, on rectangular discretizations, we only assume that neighboring cells in $\T$ have comparable size,
as in \cite{DL:05,NOS}. This mild assumption enables us also to obtain anisotropic error estimates 
for domains that can be decomposed into $n$--rectangles. 
These estimates are derived in the weighted $W_p^1$-semi-norm
and the weighted $L^p$-norm, the latter being a new result
even for the unweighted setting.
For $m=0,1$, we also derive interpolation estimates in
the space $W^m_q(\rho,\Omega)$ when the smoothness is measured in the space $W_p^{m+1}(\omega,\Omega)$, with different
weights $\omega \neq \rho$ and Lebesgue exponents $ 1< p \leq q$, provided $W_p^{m+1}(\omega,\Omega) \hookrightarrow W^m_q(\rho,\Omega)$.

The outline of this paper is as follows. In \S~\ref{sub:sec:notations} we introduce some terminology used 
throughout this work. In \S~\ref{sub:sec:weighted_spaces}, we recall the definition of a Muckenhoupt class, 
weighted Sobolev spaces and some of their properties. Section~\ref{sec:poincare} is dedicated to an important 
weighted $L^p$-based Poincar\'e inequality over star-shaped domains and domains that can be written
as the finite union of star-shaped domains. In section~\ref{sec:intweighted}, we 
consider an averaged version of the Taylor polynomial, and we 
develop a constructive theory of piecewise polynomial interpolation in
weighted Sobolev spaces with Muckenhoupt weights. 
We discuss the quasi-interpolation operator $\Pi_\T$ and its properties in 
section~\ref{sec:int}. We derive optimal approximation properties in the weighted 
$W_p^k$-seminorm for simplicial triangulations in \S~\ref{subsec:intP}.
In \S~\ref{subsec:intQ} we derive anisotropic error estimates 
on rectangular discretizations for a $\Q_1$ quasi-interpolant operator
assuming that $\Omega$ is an $n$-rectangle.
Section~\ref{sub:sec:diff_weight} is devoted to derive
optimal and local interpolation estimates for different metrics 
(\ie $p \leq q$, $\omega \neq \rho$). Finally, in 
section~\ref{sec:applications} we present applications of our interpolation theory
to nonuniformly elliptic equations \eqref{weighted_second},
elliptic equations with singular sources,
and fractional powers of elliptic operators.

\section{Notation and preliminaries}
\subsection{Notation}
\label{sub:sec:notations}

Throughout this work, $\Omega$ is an open, bounded and connected subset of $\R^n$, with $n\geq1$.
The boundary of $\Omega$ is denoted by $\partial\Omega$. Unless specified otherwise, we will assume
that $\partial\Omega$ is Lipschitz.

The set of locally integrable functions on $\Omega$ is denoted by $\Loneloc(\Omega)$.
The Lebesgue measure of a measurable subset $E \subset \R^n$ is denoted by $|E|$. The mean value of a locally 
integrable function $f$ over a set $E$ is
\[
\fint_{E} f \diff x = \frac{1}{|E|} \int_{E} f \diff x.
\]

For a multi-index $\kappa = (\kappa_1,\dots,\kappa_n) \in \mathbb{N}^n$ we denote
its length by $|\kappa| = \kappa_1 + \cdots + \kappa_n$, and, if $x \in \R^n$, we set
$x^{\kappa} = x_1^{\kappa_1}\dots x_n^{\kappa_n} \in \R$, and
\[
 D^{\kappa} = \frac{\partial^{\kappa_1}}{\partial x_1^{\kappa_1}}
 \dots \frac{\partial^{\kappa_n}}{\partial x_n^{\kappa_n}}.
\]

Given $ p\in (1 ,\infty)$, we 
denote by $p'$ the real number such that $1/p + 1/p' = 1$, \ie $p' = p/(p-1)$.

Let $ \gamma, z  \in \R^n$, the binary operation $ \circ : \R^n \times \R^{n} \rightarrow \R^{n}$
is defined by 
\begin{equation}
\label{eq:defcirc}
  \gamma \circ z = (\gamma_1 z_1, \gamma_2 z_2,\cdots, \gamma_n z_n) \in \R^{n}.
\end{equation}

If $X$ and $Y$ are topological vector spaces, we write $X \hookrightarrow Y$
to denote that $X$ is continuously embedded in $Y$. We denote by $X'$ the dual of $X$.
If $X$ is normed, we denote by $\|\cdot\|_X$ its norm.
The relation $a \lesssim b$ indicates that $a \leq Cb$, with a constant $C$ that does not
depend on either $a$ or $b$, the value of $C$ might change at each occurrence. 

\subsection{Weighted Sobolev spaces}
\label{sub:sec:weighted_spaces}
We now introduce the class of Muckenhoupt weighted Sobolev spaces and refer to
\cite{Javier,FKS:82,HKM,Kufner,Turesson} for details. We start with the definition of a weight.

\begin{definition}[weight]
\label{weight} 
A weight is a function $\omega \in \Loneloc(\R^n)$ such that $\omega(x) > 0$ for a.e.~$x \in \R^n$.
\end{definition}

Every weight induces a measure, with density $\omega \diff x$, over the Borel sets of $\R^n$.
For simplicity, this measure will also be denoted by $\omega$. For a Borel set $E \subset \R^n$
we define $\omega(E) = \int_{E}\omega \diff x$ .

We recall the definition of Muckenhoupt classes; see \cite{Javier,FKS:82,Muckenhoupt,Turesson}.

\begin{definition}[Muckenhoupt class $A_p$]
 \label{def:Muckenhoupt}
Let $\omega$ be a weight and $1 < p < \infty$. We say $\omega \in A_p(\R^n)$
if there exists a positive constant $C_{p,\omega}$ such that
\begin{equation}
  \label{A_pclass}
  \sup_{B} \left( \fint_{B} \omega \right)
            \left( \fint_{B} \omega^{1/(1-p)} \right)^{p-1}  = C_{p,\omega} < \infty,
\end{equation}
where the supremum is taken over all balls $B$ in $\R^n$.
In addition,
\[
  A_\infty(\R^n ) = \bigcup_{p>1} A_p(\R^n), \qquad A_1(\R^n) = \bigcap_{p>1} A_p(\R^n).
\]
If $\omega$ belongs to the Muckenhoupt class $A_p(\R^n)$, we say that $\omega$ is an $A_p$-weight, and
we call the constant $C_{p,\omega}$ in \eqref{A_pclass} the $A_p$-constant of $\omega$. 
\end{definition}

A classical example is the function $|x|^{\gamma}$, which is an $A_p$-weight if and only if 
$-n< \gamma < n(p-1)$. Another important example is $d(x)= \textup{\textsf{d}}(x,\partial \Omega)^{\alpha}$,
where for $x \in \Omega$, $\textup{\textsf{d}}(x,\partial \Omega)$ denotes the distance 
from the point $x$ to the boundary
$\partial \Omega$. The function $d$ belongs to $A_2(\R^n)$ if and only if $-n<\alpha<n$. This function
is used to define weighted Sobolev spaces which are important to study
Poisson problems with singular sources; see \cite{AGM,DAngelo:SINUM}.

Throughout this work, we shall use some properties of the $A_p$-weights which, for completeness, we state and 
prove below.

\begin{proposition}[properties of the $A_p$-class]
\label{pro:properties}
Let $1<p<\infty$, and $\omega \in A_p(\R^n)$. Then, we have the following properties:
\begin{enumerate}[(i)]
 \item \label{i} $\omega^{-1/(p-1)} \in \Loneloc(\R^n)$.
 \item \label{ii} $C_{p,\omega} \geq 1$.
 \item \label{iii} If $1 < p < r < \infty$, then $A_p(\R^n) \subset A_r(\R^n)$, and 
 $C_{r,\omega} \leq C_{p,\omega}$.
 \item \label{iv} $\omega^{-1/(p-1)} \in A_{p'}(\R^n)$ and, conversely,
  $\omega^{-1/(p'-1)} \in A_{p}(\R^n)$. Moreover, 
  \[ 
  C_{p',\omega^{-1/(p-1)}} = C_{p,\omega}^{1/(p-1)}. 
  \]
 \item \label{v} The $A_p$-condition is invariant under translations and isotropic dilations, \ie the weights
 $x \mapsto \omega(x+\mathbf{b})$ and $x \mapsto \omega(\mathbf{A}x)$, with $\mathbf{b} \in \R^n$
 and $\mathbf{A}=a \cdot \mathbf{I}$ with $a \in \R$, both belong to $A_p(\R^n)$ with the same 
 $A_p$-constant as $\omega$.
\end{enumerate}
\end{proposition}
\begin{proof}
Properties \eqref{i} and \eqref{iv} follow directly from the definition of the  Muckenhoupt class $A_p(\R^n)$ 
given in \eqref{A_pclass}. By writing $1=\omega^{1/p} \omega^{-1/p}$ and the H{\"o}lder inequality, we obtain 
that for every ball $B \subset \R^n$,
\[
1 = \fint_{B} \omega^{1/p}\omega^{-1/p} \leq 
\left( \fint_{B} \omega \right)^{1/p} \left( \fint_{B} \omega^{-1/(p-1)}\right)^{(p-1)/p},
\]
which proves \eqref{ii}.
Using the H{\"o}lder inequality again, we obtain
\[
 \left( \fint_{B} \omega^{1/(1-r)} \right)^{r-1}
  \leq 
  \left( \fint_{B} \omega^{1/(1-p)} \right)^{p-1},
\]
which implies \eqref{iii}.
Finally, to prove property \eqref{v} we denote $\bar{\omega}(x) = \omega(\mathbf{A}x+\mathbf{b})$,
and let $B_r$ be a ball of radius $r$ in $\R^n$. Using the change of variables $y = \mathbf{A}x+\mathbf{b}$,
we obtain
\begin{eqnarray}
  \fint_{B_r} \bar\omega(x) \diff x
= \frac{1}{a^n |B_{r}|}\int_{B_{ar}} \omega (y) \diff y,
\end{eqnarray}
which, since $a^n |B_{r}| = |B_{ar}|$, proves \eqref{v}.
\end{proof}

From the $A_p$-condition and H{\"o}lder's inequality follows that an $A_p$-weight satisfies the so-called 
\emph{strong doubling property}. The proof of this fact is standard and presented here for completeness; see 
\cite[Proposition 1.2.7]{Turesson} for more details.

\begin{proposition}[strong doubling property]
\label{pro:double}
 Let $\omega \in A_p(\R^n)$ with $1< p < \infty$ and let $E \subset \R^n$ be a measurable
subset of a ball $B \subset \R^n$. Then
\begin{equation}
\label{strong_double}
 \omega(B) \leq C_{p,\omega} \left(\frac{|B|}{|E|} \right)^p \omega(E).
\end{equation}
\end{proposition}
\begin{proof}
Since $E \subset \R^n$ is measurable, we have that
\begin{align*}
 |E|  & \leq \left( \int_{E} \omega \diff x \right)^{1/p}
            \left( \int_{E} \omega^{-\poverp} \diff x \right)^{1/p'}
            \leq \omega(E)^{1/p} |B|^{1/p'} \left( \fint_{B} \omega^{-\poverp}\right)^{1/p'} \\
         &   \leq C_{p,\omega}^{1/p} \, \omega(E)^{1/p} |B|^{1/p'} \left( \fint_{B} \omega\right)^{-1/p} 
            = C_{p,\omega}^{1/p} \left( \frac{\omega(E)}{\omega(B)} \right)^{1/p} |B|.
\end{align*}
This completes the proof.
\end{proof}

In particular, every $A_p$-weight satisfies a \emph{doubling property}, \ie there exists a positive constant 
$C$ such that
\begin{equation}
\label{double}
 \omega(B_{2r}) \leq C \omega(B_r).
\end{equation}
for every ball $B_r \subset \R^n$. The infimum over all constants $C$, for which
\eqref{double} holds, is called the \emph{doubling constant} of $\omega$.
The class of $A_p$-weights was introduced by B.~Muckenhoupt \cite{Muckenhoupt}, who
proved that the $A_p$-weights are precisely those for which the Hardy-Littlewood maximal operator
is bounded from $L^p(\omega,\R^n)$ to $L^p(\omega,\R^n)$, when $1 < p < \infty$.
We now define weighted Lebesgue spaces as follows.

\begin{definition}[weighted Lebesgue spaces]
\label{Lp_weighted} 
Let $\omega \in A_p$, and let $\Omega \subset \R^n$ be an open and bounded domain.
For $1< p < \infty$, we define the weighted Lebesgue space
$L^p(\omega, \Omega)$ as the set of measurable
functions $u$ on $\Omega$ equipped with the norm
\begin{equation}
 \| u \|_{L^p(\omega, \Omega)} = \left( \int_{\Omega} |u|^p \omega \right)^{1/p}. 
\end{equation}
\end{definition}

An immediate consequence of $\omega \in A_p(\R^n)$ is that  functions in $L^p(\omega,\Omega)$ are locally 
summable which, in fact, only requires that $\omega^{-1/(p-1)} \in \Loneloc(\R^n)$.
 
\begin{proposition}[$L^p(\omega,\Omega)\subset \Loneloc(\Omega)$]
\label{pro:loc_int}
Let $\Omega$ be an open set, $1 < p < \infty$ and $\omega$ be a weight such that 
$\omega^{-1/(p-1)} \in \Loneloc(\Omega)$.
Then, $L^p(\omega, \Omega) \subset \Loneloc(\Omega)$.
\end{proposition}
\begin{proof}
Let $u \in L^p(\omega, \Omega)$, and let $B \subset \Omega$ be a ball. By H{\"o}lder's
inequality, we have
 \[
   \int_{B} |u| = 
  \int_{B} |u| \omega^{1/p} \omega^{-1/p}
  \leq \left( \int_{B} |u|^p \omega \right)^{1/p}
  \left( \int_{B} \omega^{-1/(p-1)} \right)^{(p-1)/p}
  \lesssim \| u \|_{L^p(\omega, \Omega)},
\]
which concludes the proof.
\end{proof}

Notice that when $\Omega$ is bounded we have $L^p(\omega, \Omega) \hookrightarrow L^1(\Omega)$.
In particular, Proposition~\ref{pro:loc_int} shows that it makes sense to talk about weak 
derivatives of functions in $L^p(\omega, \Omega)$. We define weighted Sobolev spaces as follows.

\begin{definition}[weighted Sobolev spaces]
\label{H_1_weighted} 
Let $\omega$ be an $A_p$-weight with $1< p < \infty$, $\Omega \subset \R^n$ be an open and bounded domain
and $m \in \mathbb{N}$. The weighted Sobolev space $W^m_p(\omega, \Omega)$ is the set of functions
$u \in L^p(\omega,\Omega)$ such that for any multi-index $\kappa$
with $|\kappa| \leq m$, the weak derivatives
$D^{\kappa} u \in L^p(\omega, \Omega)$, with seminorm and norm
\begin{equation*}
 |u|_{W^m_p(\omega, \Omega)} = \left( \sum_{|\kappa| = m } 
 \|D^{\kappa}u\|_{L^p(\omega, \Omega)}^p \right)^{1/p},
\quad
 \| u \|_{W^m_p(\omega, \Omega)} = \left( \sum_{j \leq m } 
 |u|_{W^j_p(\omega, \Omega)}^p \right)^{1/p},
\end{equation*}
respectively. We also define $\Wmpo(\omega, \Omega)$ as the closure of $C_0^{\infty}(\Omega)$
in $W^m_p(\omega, \Omega)$.
\end{definition}

Without any restriction on the weight $\omega$, the space $W^m_p(\omega, \Omega)$
may not be complete. However, when $\omega^{-1/(p-1)}$ is locally integrable in $\R^n$, $W^m_p(\omega,\Omega)$ is a Banach space; 
see \cite{KO84}. Properties of weighted Sobolev spaces can be found in classical refe\-rences
like \cite{HKM, Kufner,Turesson}. It is remarkable that most of the properties of classical Sobolev spaces
have a weighted counterpart and it is more so that this is not because of the specific form of the
weight but rather due to the fact that the weight $\omega$ belongs to the 
Muckenhoupt class $A_p$; see \cite{FKS:82,GU,Muckenhoupt}.
In particular, we have the following 
results (cf.~\cite[Proposition 2.1.2, Corollary 2.1.6]{Turesson} and \cite[Theorem~1]{GU}) .

\begin{proposition}[properties of weighted Sobolev spaces]
\label{PR:banach}
Let $\Omega \subset \R^n$ be an open and bounded domain, $1 < p < \infty$, $\omega \in A_p(\R^n)$
and $m \in\mathbb{N}$. The spaces 
$W^m_p(\omega, \Omega)$ and $\Wmpo(\omega, \Omega)$
are complete, and $W^m_p(\omega, \Omega) \cap C^{\infty}(\Omega)$ is dense in $W^m_p(\omega, \Omega)$.
\end{proposition}

\section{A weighted Poincar\'e inequality}
\label{sec:poincare}

In order to obtain interpolation error estimates in 
$L^p(\omega,\Omega)$ and $W^1_p(\omega,\Omega)$, it is instrumental to
have a weighted Poincar\'e-like inequality \cite{DL:05,NOS}.
A pioneering reference is the work by Fabes, Kenig and Serapioni \cite{FKS:82}, which shows
that, when the domain is a ball and the weight belongs to $A_p$ with $1 < p < \infty$,
a weighted Poincar\'e inequality holds \cite[Theorem~1.3 and Theorem~1.5]{FKS:82}.
For generalizations of this result see \cite{FGW:94, HK:00}.
For a star-shaped domain, and a specific $A_2$-weight, we have proved a weighted Poincar\'e inequality \cite[Lemma 4.3]{NOS}.
In this section we extend this result to a general exponent $p$ and a general weight $\omega \in A_p(\R^n)$.
Our proof is constructive and not based on a compactness argument. This allows us to trace the dependence of 
the stability constant on the domain geometry.

\begin{lemma}[weighted Poincar\'e inequality I]
\label{le:Poincareweighted}
Let $S \subset \R^n$ be bounded, star-shaped with respect to a ball $\hat{B}$,
with $\diam S \approx 1$.
Let $\chi$ be a continuous function on $S$ with
$\| \chi \|_{L^1(S)} = 1$. Given 
$\omega \in A_p(\R^n)$,
we define
$\mu(x)= \omega(\mathbf{A}x+\mathbf{b})$,
for  $\mathbf{b} \in \R^n$ and $\mathbf{A}=a \cdot \mathbf{I}$, with $a \in \R$.
If $v \in W^1_p(\mu,S)$ is such that $\int_{S} \chi v = 0$, then 
\begin{equation}
\label{Poincareweighted}
  \| v \|_{L^p(\mu,S)} \lesssim 
\| \nabla v \|_{ L^p(\mu,S) },
\end{equation}
where the hidden constant depends only on $\chi$, $C_{p,\omega}$ and the radius $\hat{r}$ of $\hat{B}$, but
is independent of $\mathbf{A}$ and $\mathbf{b}$.
\end{lemma}
\begin{proof}
Property \eqref{v} of Proposition~\ref{pro:properties} shows that $\mu \in A_p(\R^n)$ and 
$C_{\mu,p} = C_{\omega,p}$. Given $v \in  W^1_p(\mu,S)$, we define
\[
  \tilde{v} = \signum(v) |v|^{p-1}\mu - \left(\int_S \signum(v) |v|^{p-1}\mu\right)\chi.
\]
H{\"o}lder's inequality yields
\begin{equation}
  \int_S \mu |v|^{p-1}  =  \int_S \mu^{1/p'} |v|^{p-1}  \mu^{1/p}  
  \leq  \left( \int_S \mu |v|^p \right)^{1/p'}  \left( \int_S \mu \right)^{1/p}
  \lesssim  \| v\|_{L^p(\mu,S)}^{p-1},
\label{Poincareaux1}
\end{equation}
which implies that $\tilde{v} \in L^1(S)$ and 
$\| \tilde{v} \|_{L^1(S)} \lesssim \| v\|_{L^p(\mu,S)}^{p-1}$.
Notice, in addition, that since $\int_{S} \chi=1$, the function $\tilde{v}$ has vanishing mean value. 

Given $1 < p < \infty$, we define $q = -p'/p$, and we notice that $q + p' = 1$ and $p'(p-1) = p$.
We estimate $\| \tilde{v} \|_{L^{p'}(\mu^{q},S)}$ as follows:
\begin{equation*}
  \begin{aligned}
    \left(\int_{S} \mu^{q}|\tilde{v}|^{p'} \right)^{1/p'} &=  \left(\int_{S} \mu^{q}
    \left| \signum(v) |v|^{p-1}\mu - \left(\int_{S} \signum(v) |v|^{p-1}\mu \right)\chi
    \right|^{p'} \right)^{1/p'}
    \\
    & \leq \left(\int_{S} \mu^{q+p'}|v|^{p'(p-1)}\right)^{1/p'}
    +
    \left(\int_{S} |v|^{p-1}\mu \right)
    \| \chi\|_{L^{p'}(\mu^q,S)}
    \\
    & \lesssim \| v\|_{L^p(\mu,S)}^{p-1},
  \end{aligned}
\label{aux1:2}
\end{equation*}
where we have used \eqref{Poincareaux1} together with the fact that $\mu \in A_p(\R^n)$ implies  
$\mu^q \in \Loneloc(\R^n)$ (see Proposition~\ref{pro:properties} \eqref{i}), whence
$\| \chi\|_{L^{p'}(\mu^q,S)} \leq \|\chi \|_{L^\infty(S)}\mu^q(S)^{1/p'} \lesssim 1$.

Properties $\mu^q \in A_{p'}(\R^n)$, that $S$ is star-shaped with respect to $\hat B$
and $\tilde{v} \in L^{p'}(\mu^{q},S)$ has vanishing mean value, suffice for the existence of a vector field 
$\textbf{u} \in \Wpp(\mu^q,S)$ satisfying
\[
 \DIV ~\bu = \tilde{v},
\]
and, 
\begin{equation}
\label{Fdivweighted}
  \|\nabla  \textbf{u} \|_{ L^{p'}(\mu^q,S)}  
  \lesssim 
\| \tilde{v} \|_{L^{p'}( \mu^q, S )},
\end{equation}
where the hidden constant depends on $C_{p',\mu^q}$ and the radius $r$ of $\hat B$; see \cite[Theorem~3.1]{DL:10}.

Finally, since $\int_{S} \chi v = 0$, the definition of $\tilde{v}$ implies
\begin{equation*}
\label{aux1:1}
  \| v \|^p_{L^p(\mu,S)} = \int_{S} v \tilde{v} +
  \left(\int \signum(v) |v|^{p-1}\mu \right)\int_{S}\chi v =
  \int_{S} v \tilde{v}.
\end{equation*}
Replacing $\tilde{v}$ by $-\DIV~\bu$,
integrating by parts and using \eqref{Fdivweighted}, we conclude
\begin{align*}
\| v \|^p_{L^p(\mu,S)}  
= \int_{S} \nabla v \cdot \bu
& \leq \left( \int_{S} \mu |\nabla v|^p \right)^{1/p}
\left( \int_{S} \mu^q |\bu|^{p'} \right)^{1/p'}
\\
& \lesssim 
\| \nabla v\|_{ L^p(\mu,S)} 
\| \tilde{v} \|_{L^{p'}(\mu^{q},S)}.
\end{align*}
Invoking $\| \tilde{v}\|_{L^{p'}(\mu^q,S)}\lesssim \| v\|_{L^p(\mu,S)}^{p-1}$ yields the desired inequality.
\end{proof}

In section~\ref{sec:int} we construct an interpolation operator based on local averages. Consequently,
the error estimates on an element $T$ depend on the behavior of the function over a so-called
\emph{patch} of $T$, which is not necessarily star shaped. Then, we need to relax the geometric 
assumptions on the domain $S$ and let the vanishing mean property hold just in a
subdomain. The following result is an adaptation of \cite[Corollary 4.4]{NOS}.

\begin{corollary}[weighted Poincar\'e inequality II]
\label{C:Poincareweighted-2}
Let $S=\cup_{i=1}^N S_i \subset\R^{n}$ be a connected domain and each $S_i$ be star-shaped with respect to a 
ball $B_i$. Let $\chi_i\in C^0(\bar S_i)$ and $\mu$ be as in Lemma~\ref{le:Poincareweighted}. If 
$v \in W^1_p(\mu, S)$ and $v_i =\int_{S_i} v\chi_i$, then
\begin{equation}
\label{Poincareweighted-2}
  \|v - v_i\|_{L^p(\mu,S)} \lesssim 
\| \nabla v \|_{L^p(\mu,S)}
\qquad \forall 1\le i\le N,
\end{equation}
where the hidden constant depends on $\{\chi_i\}_{i=1}^N$, the radii $r_i$ of $B_i$, and the amount of overlap 
between the subdomains $\{S_i\}_{i=1}^N$, but is independent of $\mathbf{A}$ and $\mathbf{b}$.
\end{corollary}
\begin{proof}
This is an easy consequence of Lemma~\ref{le:Poincareweighted} and \cite[Theorem 7.1]{DupScott}.
For completeness, we sketch the proof. It suffices to deal with two subdomains $S_1,S_2$ and the overlapping 
region $D = S_1\cap S_2$. We start from
\[
  \|v - v_1\|_{L^p(\mu,S_2)} \le \|v - v_2\|_{L^p(\mu,S_2)} + \|v_1-v_2\|_{L^p(\mu,S_2)}.
\]
Since $v_1$ and $v_2$ are constant
\[
  \|v_1 - v_2\|_{L^p(\mu,S_2)} = \left( \frac{\mu(S_2)}{\mu(D)} \right)^{1/p} \|v_1 - v_2\|_{L^p(\mu,D)},
\]
which together with
\[
  \|v_1 - v_2\|_{L^p(\mu,D)} \leq \|v-v_1\|_{L^p(\mu,S_1)} + \|v-v_2\|_{L^p(\mu,S_2)},
\] 
and \eqref{Poincareweighted} imply 
$\|v - v_1\|_{L^p(\mu,S_2)} \lesssim 
\|\nabla v\|_{L^p(\mu,{S_1 \cup S_2})] }
$.
This and \eqref{Poincareweighted} give \eqref{Poincareweighted-2} for $i=1$,
with a stability constant depending on the ratio  $\frac{\mu(S_2)}{\mu(D)}$.
\end{proof}

\section{Approximation theory in weighted Sobolev spaces}
\label{sec:intweighted}

In this section, we introduce an averaged version of the Taylor polynomial
and study its approximation properties in Muckenhoupt weighted Sobolev spaces.
Our results are optimal and are used to obtain error estimates for the quasi-interpolation operator defined in section~\ref{sec:int} on 
simplicial and rectangular discretizations. 
The interpolation operator is built on local averages over stars, and so is similar to the one introduced in \cite{DupScott}. The main difference
is that it is directly defined on the given mesh instead of using a reference element.
This idea is fundamental in order to relax the regularity assumptions on the elements, which is what allows us to derive the anisotropic estimates on rectangular elements presented in \S~\ref{subsec:intQ}.

\subsection{Discretization}
\label{sub:sec:preliminars}

We start with some terminology
and describe the construction of the underlying finite element spaces. 
In order to avoid technical 
difficulties we shall assume $\partial \Omega$ is polyhedral. We denote by $\T = \{T\}$ a partition, or mesh, of 
$\Omega$ into elements $T$ (simplices or cubes) such that
\[
  \bar\Omega = \bigcup_{T \in \T} T, \qquad
  |\Omega| = \sum_{T \in \T} |T|.
\]
The mesh $\T$ is assumed to be conforming or compatible: the intersection of any two elements 
is either empty or a common lower dimensional element. We denote by $\Tr$ a collection of conforming meshes, which
are \emph{shape regular} i.e., there 
exists a constant $\sigma > 1$ such that, for all $\T \in \Tr,$
\begin{equation}
\label{shape_reg}
 \max \left\{ \sigma_T : T \in \T \right\} \leq \sigma,
\end{equation}
where $\sigma_T = h_T/\rho_T$ is the shape coefficient of $T$. In the case of simplices, 
$h_T = \diam(T)$ and $\rho_T$ is the diameter of the sphere 
inscribed in $T$; see, for instance, \cite{BrennerScott}. For the definition of
$h_T$ and $\rho_T$ in the case of $n$-rectangles see \cite{Ciarletbook}.

In \S~\ref{subsec:intQ}, we consider
rectangular discretizations of the domain $\Omega = (0,1)^n$ which satisfy 
a weaker regularity assumption and thus allow 
for anisotropy in each coordinate direction (\cf \cite{DL:05}).

Given a mesh $\T \in \Tr$,
we define the finite element
space of continuous
piecewise polynomials of degree $m \geq 1$
\begin{equation}
  \V(\T) = \left\{
            W \in \C^0( \bar \Omega): W_{|T} \in \mathcal{P}(T) \ \forall T \in \T, \
            W_{|\partial \Omega} = 0
          \right\},
\label{eq:defFESpace}
\end{equation}
where, for a simplicial element $T$, $\mathcal{P}(T)$ corresponds to $\mathbb{P}_m$ --- the space of polynomials of total degree at most $m$.
If $T$ is an $n$-rectangle, then 
$\mathcal{P}(T)$ stands for $\mathbb{Q}_m$ --- the space of polynomials of degree not larger than $m$ in each variable.

Given an element $T \in \T$, we denote by $\N(T)$ and $\Nin(T)$ the set of 
nodes and interior nodes of $T$, respectively. We
set $\N(\T) := \cup_{T \in \T} \N(T) $ and $\Nin(\T) := \N(\T) \cap \partial \Omega$.
Then, any discrete function $V \in \V(\T)$ is characterized by its 
nodal values on the set $\Nin(\T)$. 
Moreover, the functions $\phi_z \in \V(\T)$, $z \in \Nin(\T)$, such that
$\phi_z(y)=\delta_{yz}$ for all $y \in \N(\T)$ are 
the canonical basis of $\V(\T)$, and
\[
 V = \sum_{z \in \Ninn(\T)} V(z) \phi_z.
\]
The functions 
$\{ \phi_z\}_{z \in \Ninn(\T)}$ are the so called \emph{shape functions}.

Given $z \in \N(\T)$, the \emph{star} or patch around $z$ is
$
  S_{z} := \bigcup_{z \in T} T,
$
and, for $T \in \T$, its \emph{patch} is
$
  S_T := \bigcup_{z \in T} S_z.
$
For each $z \in \N(\T)$, we define 
$h_{z} := \min\{h_{T}: z \in T\}$.

\subsection{The averaged interpolation operator}
\label{sub:sec:average}

We now develop an approximation theory in Muckenhoupt weighted Sobolev spaces,
which is instrumental in section \ref{sec:int}. We define
an averaged Taylor polynomial, built on local averages
over stars and thus well defined for $L^p(\omega,\Omega)$-functions.
Exploiting the weighted Poincar\'e inequality derived in section~\ref{sec:poincare}, 
we show optimal error estimates for constant and linear approximations.
These results are the basis to extend these estimates to 
any polynomial degree via a simple induction argument in section 
\ref{sub:sec:induction}.

Let $\psi \in C^{\infty}(\R^{n})$ be such that $\int \psi = 1$ and $\supp \psi \subset B$, where 
$B$ denotes the ball in $\R^n$ of radius $r = r(\sigma)$ and centered at zero.
For $z \in \Nin(\T)$, we define the rescaled smooth functions
\begin{equation}
\label{psi_z}
\psi_{z}(x) =  \frac{(m+1)^n}{ h_{z}^n } \psi
      \left(\frac{(m+1)(z-x)}{h_{z}}\right), 
\end{equation}
where $m\geq 0$ is the polynomial degree. The scaling of $\psi_z$ involving the factor $m+1$ guarantees the property
\[
 \supp \psi_{z} \subset S_{z}
\]
for all nodes $z \in \Nin(\T)$ (not just the interior vertices of $\T$) provided $r$ is suitable chosen. This is because the distance from $z$ to $\partial S_z$ is proportional to $h_z/(m+1)$ for shape regular meshes.

Given a smooth function $v$, we denote by 
$P^m v(x,y)$ the Taylor polynomial of order $m$ in the variable $y$ about the point 
$x$, \ie
\begin{equation}
\label{taylor}
 P^m v (x,y)  = \sum_{| \alpha | \leq m} \frac{1}{\alpha!} D^{\alpha}v(x)(y-x)^{\alpha}.
\end{equation}
For $z \in \Nin(\T)$, and $v \in W^m_p(\omega,\Omega)$, 
we define the corresponding \emph{averaged} Taylor polynomial
of order $m$ of $v$ about the node $z$ as 
\begin{equation}
\label{averaged_taylor}
  Q^m_{z} v(y) = \int P^mv(x,y) \psi_{z}(x) \diff x.
\end{equation}
Integration by parts shows that $Q^m_{z} v$ is well-defined for
functions in $L^1(\Omega)$ \cite[Proposition 4.1.12]{BrennerScott}.
Proposition~\ref{pro:loc_int} then
allows us to conclude that \eqref{averaged_taylor} is well defined for
$v \in L^p(\omega,\Omega)$. 
Since $\supp \psi_{z} \subset S_{z}$, 
the integral appearing in \eqref{averaged_taylor} can be also written over $S_{z}$.
Moreover, we have the following properties of $Q_{z}^m v$:
\begin{enumerate}[$\bullet$]
\item $Q^m_{z} v$ is a polynomial of degree less or equal than $m$ in
the variable $y$ (\cf \cite[Proposition 4.1.9]{BrennerScott}).
\item $Q^m_{z} v = Q^m_{z} Q^m_{z} v$, \ie $Q^m_{z}$ is
    invariant over $\mathbb{P}_m$.
\item For any $\alpha$
such that $|\alpha| \leq m $,
\begin{equation}
\label{prop_der}
 D^{\alpha} Q_{z}^m v = Q_{z}^{m-|\alpha|} D^{\alpha} v
\qquad \forall v \in W_1^{|\alpha|}(B),
\end{equation}
(\cf \cite[Proposition 4.1.17]{BrennerScott}).
As a consequence of $\omega \in A_p(\R^n)$, together with Proposition \ref{pro:loc_int}, we have that \eqref{prop_der} holds for 
$v$ in $W_1^{|\alpha|}(\omega,B)$.
\end{enumerate}

The following stability result is important in the subsequent analysis.

\begin{lemma}[stability of $Q_z^m$]
\label{LM:Qm_stab}
Let $\omega \in A_{p}(\R^n)$ and $z \in \Nin(\T)$. If $v \in W_p^{k}(\omega,S_z)$, 
with $ 0 \leq k \leq m$, we have the following stability result
\begin{equation}
\label{Qm_stab}
 \| Q_z^m v\|_{L^{\infty}(S_z)} \lesssim h_z^{-n }\| 1 \|_{L^{p'}(\omega^{-p'/p},S_z)}
\sum_{ l = 0}^{k} h_z^{l}
|  v |_{W_p^{l}(\omega,S_z)}.
\end{equation}
\end{lemma}
\begin{proof}
Using the definition of the averaged Taylor polynomial \eqref{averaged_taylor}, 
we arrive at
\[
\| Q_z^m v\|_{L^{\infty}(S_z)}  \lesssim 
\sum_{|\alpha| \leq m} \left\| \int_{S_z}  D^\alpha v(x)
 (y-x)^{\alpha} \psi_z(x) \diff x
 \right\|_{L^{\infty}(S_z)} .
\]
This implies estimate \eqref{Qm_stab} if $k=m$. Otherwise,
integration by parts on the higher derivatives
$D^{\alpha} v$ with $k < |\alpha| \leq m$, $\psi_z = 0$
on $\partial S_z$, the fact that $D^{\alpha} \psi$ is uniformly bounded on $\mathbb{R}^n$, the estimate $|y-x| \lesssim h_z$ for all $x,y
\in S_z$, together with H{\"o}lder's inequality, yield \eqref{Qm_stab}.
\end{proof}

Given $\omega \in A_p(\R^n)$ and $v \in W_p^{m+1}(\omega,\Omega)$ with $m \geq 0$,
in the next section we derive approximation properties of the averaged
Taylor polynomial $Q_{z}^m v$
in the weighted $W_p^k(\omega,\Omega)$-norm, with $0 \leq k \leq m$, via 
a weighted Poincar\'e inequality and a simple induction argument.
Consequently, we must first 
study the approximation properties of $Q_{z}^0 v$,
the weighted average of $v \in L^p(\omega,\Omega)$,
which for $z \in \Nin(\T)$ reads
\begin{equation}
\label{p0average}
    Q^0_{z}v = \int_{S_z} v(x) \psi_{z}(x) \diff x.
\end{equation} 

\subsection{Weighted $L^p$-based error estimates}
\label{sub:sec:Lp_estimates}

We start by adapting
the proofs of \cite[Lemma~2.3]{DL:05} and \cite[Lemma 4.5]{NOS} to obtain local approximation 
estimates in the weighted $L^p$-norm for the polynomials $Q^0_{z}v$ and $Q^1_{z}v$. 

\begin{lemma}[weighted $L^p$-based error estimates]
\label{LM:approximation}
Let $z \in \Nin(\T)$. If $v \in W^1_p(\omega,S_z)$, then we have
\begin{equation}
\label{v-Q_vLp}
  \|v - Q^0_{z}v \|_{L^p(\omega,S_z)}
  \lesssim   h_{z}  
  \| \nabla v\|_{L^p(\omega,S_z)}.
\end{equation}
If $v \in W^2_p(\omega,S_z)$ instead, the following estimate holds
\begin{equation}
\label{v-v_vW1p}
  \|\partial_{x_j}(v - Q^1_{z}v) \|_{L^p(\omega,S_z)}
  \lesssim h_{z} 
  \| \partial_{x_j}\nabla v\|_{L^p(\omega,S_z) },
  \end{equation}
for $j=1,\dots,n$. In both inequalities, the hidden constants depend only on $C_{p,\omega}$, $\sigma$ and $\psi$.
\end{lemma}
\begin{proof}
Define the mapping $\Fm_z: x \mapsto \bar{x}$ by
\[
  \bar{x} = \frac{z-x}{h_{z}},
\]
the star $\bar{S}_{z} = \Fm_z(S_z)$ and the function $\bar{v}(\bar x) = v(x)$.
Set
$
  \bar{Q}^0\bar{v}= \int\bar{v} \psi \diff \bar{x},
$
where $\psi$ is the smooth function introduced in section~\ref{sub:sec:average}.

Notice that $\supp \psi \subset \bar{S}_{z}$. Consequently, in the definition of $\bar Q^0 \bar v$, 
integration takes place over $\bar{S}_{z}$ only. Using the mapping $\Fm_z$, we have
\[
Q^0_{z}v = \int_{S_{z}} v \psi_z \diff x = 
\int_{\bar{S}_{z}} \bar{v} \psi \diff \bar{x}   = 
\bar{Q}^0\bar{v},
\]
and, since $\int_{\bar{S}_z} \psi \diff \bar{x} = 1$,
\begin{equation}
\label{int=0}
  \int_{\bar{S}_{z}} (\overline{v} - \bar{Q}^0\bar{v})\psi \diff \bar{x} =
  \int_{\bar{S}_{z}}\bar{v} \psi\diff \bar{x} - \bar{Q}^0\bar{v}  = 0.
\end{equation}

Define the weight $\bar{\omega}_{z} = \omega \circ \Fm_{z}^{-1}$. In light of property \eqref{v} in 
Proposition~\ref{pro:properties} we have $\bar{\omega}_{z} \in A_p(\R^n)$ and  
$C_{p,\bar{\omega}_{z}}=C_{p,\omega}$. Changing variables we get
\begin{equation}
\label{changevar1}
  \int_{S_{z}} \omega | v - Q^0_{z} v |^p \diff x=
  h_{z}^n 
  \int_{\bar{S}_{z}} \bar{\omega}_{z}
  | \bar{v} - \bar{Q}^0\bar{v} |^p \diff \bar{x} .
\end{equation}

As a consequence of the shape regularity assumption \eqref{shape_reg}, 
$\diam \bar{S}_{z} \approx 1$. Then,
in view of \eqref{int=0}, we can apply Lemma~\ref{le:Poincareweighted}
to $\bar{v} - \bar{Q}^0\bar{v}$ over $S = \bar{S}_{z}$, with $\mu = \bar{\omega}_{z}$
and $\chi = \psi$, to conclude
\[
  \| \bar{v} - \bar{Q}^0\bar{v} \|_{L^p (\bar{\omega}_{z},\bar{S}_{z})} \lesssim
  \| \bar\nabla \bar{v} \|_{ L^p(\bar{\omega}_{z},\bar{S}_{z}) },
\]
where the hidden constant depends only on $\sigma$, $C_{p,\bar\omega_z}$ and $\psi$.
Inserting this estimate into \eqref{changevar1} and changing variables with $\F_z^{-1}$ to get back to $\bar{S}_{z}$
we get \eqref{v-Q_vLp}. 

In order to prove \eqref{v-v_vW1p}, we modify $\mathcal{F}_z$ and $\bar{S}_{z}$ appropriately and define
\[
  \bar{Q}^1 \bar v(\bar{y}) = \int_{\bar{S}_{z}}
  \left( \bar{v}(\bar{x}) + \bar{\nabla} \bar{v}(\bar{x})\cdot (\bar{y} - \bar{x} ) \right) \psi(\bar{x}) \diff \bar{x},
\]
We observe that $Q^1_{z}v(y) = \bar{Q}^1\bar v(\bar{y})$,
where $Q^1_{z}v$ is defined by \eqref{averaged_taylor}. Since
$
  \partial_{\bar{y}_i} \bar{Q}^1\bar v(\bar{y}) =
  \int_{\bar{S}_{z}} \partial_{\bar{x}_i} \bar{v}(\bar{x})\psi(\bar{x}) \diff \bar{x}
$
is constant for $i \in \{1,\cdots, n\}$,
we have the vanishing mean value property
\[
  \int_{\bar{S}_{z}} \partial_{\bar{x}_i}
      \bigg( \bar{v}(\bar x)-\bar{Q}^1\bar v(\bar{x})
      \bigg) \psi(\bar x) \diff \bar x = 0.
\]
This, together with Lemma \ref{le:Poincareweighted}, leads to \eqref{v-v_vW1p}.
\end{proof}

The following result is an optimal error estimate in the $L^p$-weighted norm
for the averaged Taylor polynomial $Q_{z}^1v$, which is instrumental
to study $Q_{z}^m v$ ($m \geq 0$).

\begin{lemma}[weighted $L^p$-based error estimate for $Q_{z}^1$]
\label{LM:double_approximation}
Let $z \in \Nin(\T)$. If $v \in W^2_p(\omega,S_z)$, then
the following estimate holds
\begin{equation}
\label{v-Q_1Lp}
  \|v - Q^1_{z} v \|_{L^p(\omega,S_z)}
  \lesssim h_{z}^2 | v |_{W^2_p(\omega,S_z)},
  \end{equation}
where the hidden constant depends only on $C_{p,\omega}$, $\sigma$ and $\psi$.
\end{lemma}
\begin{proof}
Since
\begin{equation*}
v -  Q^1_{z}v = (v - Q^1_{z}v) - Q^0_{z}(v - Q^1_{z}v)
- Q^0_{z}(Q^1_{z}v-v),
\end{equation*}
and $\nabla(v-Q^1_{z}v) = \nabla v - Q^0_{z}\nabla v$ from 
\eqref{prop_der}, we can apply \eqref{v-Q_vLp} twice to obtain
\begin{align*}
\|(v - Q^1_{z}v) - Q^0_{z}(v - Q^1_{z}v)\|_{L^p(\omega,S_z)} \lesssim 
h_{z}\|\nabla(v - Q^1_{z}v) \|_{ L^p(\omega,S_z) }
\lesssim
h^2_{z} | v |_{W^2_p(\omega,S_z)}.
\end{align*}
So it remains to estimate the term 
$R^1_z(v) := Q^0_{z}( Q^1_z v -  v)$. Since
$Q_{z}^0 v = Q_{z}^0 Q_{z}^0 v $, we notice that
$R^1_z(v) = Q^0_{z} ( Q^1_z v -  Q^0_{z} v )$. Then,
using the definition of the averaged Taylor polynomial 
given by \eqref{averaged_taylor}, we have
\begin{equation*}
\label{r1}
R^1_z(v) = \int_{S_{z}} \left(  \int_{S_{z}} 
 \nabla v(x) \cdot (y-x) \psi_{z} (x) \diff x \right) \psi_{z} (y)  \diff y.
\end{equation*}
We exploit the crucial \emph{cancellation
 property} $R^1_z(p) = 0$ for all $p \in \mathbb{P}_1$ as follows:
$R^1_z(v) = R^1_z(v - Q_{z}^1 v) = 0$. 
This yields
\begin{equation*}
\|R^1_z(v)\|^p_{L^p(\omega,S_z)} = 
\int_{S_{z}}\omega\left| \int_{S_{z}} \left(  \int_{S_{z}}  \nabla ( v(x) - Q^1_{z}v(x) ) \cdot (y-x)  \psi_{z} (x) \diff x \right) \psi_{z} (y)  \diff y
\right| ^p
\end{equation*}
Applying H\"older inequality to the innermost integral $I(y)$ leads to
\begin{equation*}
|I(y)|^p \lesssim h_{z}^p 
\left( \int_{S_{z}} \omega | \nabla(v(x) - Q^1_{z} v(x)) |^p \diff x \right)
\left( \int_{S_{z}} \omega^{-p'/p} \psi_{z}(x)^{p'} \diff x
\right)^{p/p'}.
\end{equation*}
This is combined with $\int_{S_z} \psi_z(y) \diff y = 1$ and
$
  \| \psi_{z}\|_{L^{p'}(\omega^{-\poverp},S_{z})} \| 1\|_{L^p(\omega,S_{z})} \lesssim 1,
$ 
which follows from
the definition of $\psi_{z}$ and the definition \eqref{A_pclass} of
the $A_p$-class, to arrive at
\begin{equation}\label{doble_aux_2}
\|R^1_z(v)\|^p_{L^p(\omega,S_z)} \lesssim h_{z}^{2p} \int_{S_{z}} \omega |D^2 v|^p. 
\end{equation}
This yields the desired estimate \eqref{v-Q_1Lp}.
\end{proof}

\subsection{Induction argument}
\label{sub:sec:induction}
In order to derive approximation properties of the averaged Taylor polynomial
$Q_{z}^m v$ for any $m \geq 0$, we apply an induction argument. We assume the 
following estimate as \emph{induction hypothesis:}
\begin{align}
\label{induction_hypo}
\| v - Q_{z}^{m-1} v\|_{L^p(\omega,S_{z})} \lesssim h_{z}^m
| v |_{W^m_p(\omega,S_{z})}.
\end{align}
Notice that, for $m=1$, the induction hypothesis is exactly \eqref{v-v_vW1p}, 
while for $m=2$ it is given by Lemma~\ref{LM:double_approximation}. 
We have the following general result for any $m \geq 0$.

\begin{lemma}[weighted $L^p$-based error estimate for $Q_{z}^m$]
\label{LM:m_approximation}
Let $z \in \Nin(\T)$ and $m \geq 0$. If $v \in W^{m+1}_p(\omega,S_z)$, then
we have the following approximation result
\begin{equation}
\label{v-Q_mLp}
  \|v - Q^m_{z} v \|_{L^p(\omega,S_z)}
  \lesssim h_{z}^{m+1} | v |_{W_p^{m+1}(\omega,S_z)},
  \end{equation}
where the hidden constant depends only on $C_{p,\omega}$, $\sigma$,
$\psi$ and $m$.
\end{lemma}
\begin{proof}
We proceed as in the proof of Lemma~\ref{LM:double_approximation}. 
Notice, first of all, that
\begin{align*}
v - Q^m_z v = (v - Q^m_z v) - Q^{m-1}_z(v - Q^m_z v) 
- Q^{m-1}_z (Q^m_z v-v).
\end{align*}
The induction hypothesis \eqref{induction_hypo} yields
\begin{equation*}
  \| (v - Q^m_z v) - Q^{m-1}_z(v - Q^m_z v) \|_{L^p(\omega,S_z)} 
  \lesssim h_z^m | v - Q^m_z v |_{W^m_p(\omega,S_z)}.
\end{equation*}
Since $D^\alpha Q^m_z v = Q^0_z D^\alpha v$ for all $|\alpha|=m$,
according to property \eqref{prop_der}, the estimate \eqref{v-Q_vLp} yields
$| v - Q^m_z v |_{W^m_p(\omega,S_z)} \lesssim h_z | v |_{W^{m+1}_p(\omega,S_z)}$, and then
\begin{equation*}
 \| (v - Q^m_z v) - Q^{m-1}_z(v - Q^m_z v) \|_{L^p(\omega,S_z)}
  \lesssim h_z^{m+1} | v |_{W^{m+1}_p(\omega,S_z)}.
\end{equation*}

It thus remains to bound the term
\[
 R^m_z(v) := Q_{z}^{m-1} (Q_{z}^{m} v - v).
\]
Since $Q_{z}^{m-1} Q_{z}^{m-1} v = Q_{z}^{m-1} v$, writing 
$Q^m_z = Q^{m-1}_z + \sum_{|\beta|=m} T_z^\beta$ with
\[
T_z^\beta (v) = \frac1{\beta!} \int_{S_z} D^\beta v(\zeta) (x - \zeta)^\beta \psi_z(\zeta) \diff \zeta,
\]
we obtain
\[
R^m_z(v) = \sum_{|\beta|=m} Q^{m-1}_z T_z^\beta (v).
\]
This representation allows us to write
\[
 R^m_z (v)(y) = \sum_{|\alpha|<m,|\beta|=m} I_{\alpha,\beta} v (y),
\]
with
\begin{align*}
  I_{\alpha,\beta} v (y) &= \frac1{\alpha!} \int_{S_z} \psi_z(x) D_x^\alpha
  T^\beta_z v(x) (y-x)^\alpha \diff x \\
  &= \frac1{\alpha!} \int_{S_z} \psi_z(x)
  \frac1{(\beta-\alpha)!} \int_{S_z} D_\zeta^\beta v(\zeta) (x - \zeta)^{\beta-\alpha} \psi_z(\zeta) \diff \zeta (y-x)^\alpha \diff x.
\end{align*}
Finally, we notice the following \emph{cancellation property}: 
$Q^m_z p = p$ for all $p\in \mathbb{P}_m$, whence $R^m_z (p) = 0$.
Consequently $R^m_z (v) = R^m_z (v - Q^m_z v)$ implies
\begin{align*}
  \| I_{\alpha,\beta} v \|_{L^p(\omega,S_z)}^p \lesssim h_z^{mp} \int_{S_z} \omega(y)
  \left|
    \int_{S_z} \psi_z(x) \int_{S_z} D_\zeta^\beta( v 
  - Q^m_z v)(\zeta) \psi_z(\zeta) \diff \zeta \diff x 
  \right|^p \diff y .
\end{align*}
Combining the identity $ D^\beta Q_{z}^m v =  Q_{z}^0 D^\beta v$, 
with \eqref{v-Q_vLp} and the bound
\[ 
  \| \psi_{z}\|_{L^{p'}(\omega^{-\poverp},S_{z})} \| 1\|_{L^p(\omega,S_{z})} \lesssim 1,
\]
we infer that
\begin{align*}
\| R^{m}_z v \|^p_{L^p(\omega,S_{z})}  & \lesssim h_{z}^{mp}\| 1 \|^p_{L^p(\omega,S_{z})} \|D^m v - D^m Q_{z}^m v \|^p_{L^p(\omega,S_{z})}
\| \psi_{z} \|_{L^p(\omega^{-p'/p},S_{z})}^p
\\
& \lesssim h_{z}^{(m+1)p} | v |^p_{W_p^{m+1}(\omega,S_{z})}.
\end{align*}
This concludes the proof.
\end{proof}

The following corollary is a simple consequence of Lemma~\ref{LM:m_approximation}.

\begin{corollary}[weighted $W_p^k$-based error estimate for $Q_{z}^m$]
\label{CO:k_approximation}
Let $z \in \Nin(\T)$. If $v \in W^{m+1}_p(\omega,S_z)$ with $m \geq 0$, then
\begin{equation}
\label{v-Q_mWkp}
  |v - Q^m_{z} v |_{W^k_p(\omega,S_z)}
  \lesssim h_{z}^{m + 1 -k } | v |_{W_p^{m+1}(\omega,S_z)},
\quad k = 0,1,\dots,m+1,
  \end{equation}
where the hidden constant depends only on $C_{p,\omega}$, $\sigma$, $\psi$ and $m$.
\end{corollary}
\begin{proof} For $k=0$, the estimate \eqref{v-Q_mWkp} is given by Lemma \ref{LM:m_approximation}, while for $k= m +1 $,  
 \[ 
 |v - Q^m_{z} v |_{W^{m+1}_p(\omega,S_z)} = |v |_{W^{m+1}_p(\omega,S_z)}.
 \] 
For $ 0 < k < m+1$, we employ property \eqref{prop_der} of
  $D^\alpha Q_{z}^m v$ with $|\alpha|=k$ to write
\begin{equation*}
  |v - Q^m_{z} v |_{W^k_p(\omega,S_z)} 
   = \left ( \sum_{|\alpha| = k }  \| D^{\alpha} v -  Q^{m-k}_{z} D^{\alpha}v \|^p_{L^p(\omega,S_z)} \right)^{1/p}.
\end{equation*}
Therefore, applying estimate \eqref{v-Q_mLp} to 
$\| D^{\alpha} v -  Q^{m-k}_{z} D^{\alpha}v \|_{L^p(\omega,S_z)}$, we obtain
\begin{equation*}
|v - Q^m_{z} v |_{W^k_p(\omega,S_z)}
  \lesssim  h_{z}^{m+1-k} | v |_{W_p^{m+1}(\omega,S_z)},
\end{equation*}
which is the asserted estimate.
\end{proof}

\section{Weighted interpolation error estimates}
\label{sec:int}
In this section we construct a quasi-interpolation
operator $\Pi_{\T}$, based on local averages over stars.
This construction is well defined for functions in
$L^1(\Omega)$, and thus for functions in the weighted space $L^p(\omega,\Omega)$.
It is well known that this type of quasi-interpolation 
operator is important in the approximation of nonsmooth functions 
without point values because the Lagrange interpolation operator
is not even defined \cite{Clement,SZ:90}.
Moreover, averaged interpolation has better approximation
properties than the Lagrange interpolation for anisotropic
elements \cite{Acosta}. We refer the reader to \cite{NS,DL:05,NOS}
for applications of quasi-interpolation.

The construction of $\Pi_\T$ is based on the averaged Taylor
polynomial defined in \eqref{averaged_taylor}.
In \S~\ref{subsec:intP}, using the approximation estimates derived in section~\ref{sec:intweighted} together with an invariance property of $\Pi_{\T}$
over the space of polynomials, we derive optimal error estimates for 
$\Pi_{\T}$ in Muckenhoupt weighted Sobolev norms on 
simplicial discretizations. The case of rectangular discretizations
is considered in \S~\ref{subsec:intQ}.

Given $\omega \in A_p(\R^n)$ and $v \in L^p(\omega,\Omega)$,
we recall that $Q^m_z v$ is the averaged Taylor polynomial of order $m$ 
of $v$ over the node $z$; see \eqref{averaged_taylor}.
We define the quasi-interpolant $\Pi_{\T}v$ as the unique function
of $\V(\T)$ that satisfies $\Pi_{\T}v(z) = Q^m_z v(z)$ if $z \in \Nin(\T)$, and
$\Pi_{\T}v(z) = 0$ if $z \in \N(\T) \cap \partial \Omega$, \ie
\begin{equation}
\label{def:PI}
\Pi_{\T} v = \sum_{z \in \Ninn(\T)} Q^m_{z} v(z) \, \phi_z.
\end{equation}

Optimal error estimates for $\Pi_{\T}$ rely on its stability, which follows from the stability of $Q_z^m$ obtained in Lemma~\ref{LM:Qm_stab}.

\begin{lemma}[stability of $\Pi_\T$]
\label{LM:stab_PI}
Let $v \in W_p^{k}(\omega,S_T)$ with $0 \leq k \leq m+1$
and $T \in \T$. Then,
the quasi-interpolant operator $\Pi_{\T}$ defined by \eqref{def:PI} 
satisfies the following local stability bound
\begin{equation}
\label{stab_PI}
 | \Pi_{\T} v |_{W^k_p(\omega,T)} 
\lesssim  \sum_{ l = 0}^{k} h_T^{l-k}
|  v |_{W^l_p(\omega,S_T)}.
\end{equation}
\end{lemma}
\begin{proof} Using the definition of $\Pi_{\T}$ given by \eqref{def:PI}, we have
\[
 | \Pi_{\T} v |_{W^k_p(\omega,T)} 
\leq 
 \sum_{z \in \Ninn(T)} \| Q^m_{z} v \|_{L^{\infty}(S_z)} ~ | \phi_z |_{W^k_p(\omega,T)}.
\] 
We resort to Lemma \ref{LM:Qm_stab} to derive
\begin{align*}
| \Pi_{\T} v |_{W^k_p(\omega,T)} \lesssim 
\sum_{z \in \Ninn(T)} h_z^{-n} | \phi_z |_{W^k_p(\omega,T)} 
\| 1 \|_{L^{p'}(\omega^{-p'/p},S_z)} \sum_{ l = 0}^{k} h_z^{l}
|  v |_{W_p^{l}(\omega,S_z)}.
\end{align*}
Since $|D^k \phi_z| \lesssim h_z^{-k}$ on $S_T$ and $\omega \in A_p(\R^{n})$, we obtain
\[
h_z^{-n} | \phi_z |_{W^k_p(\omega,T)}
  \| 1 \|_{L^{p'}(\omega^{-\poverp},\,S_z)}
  \lesssim 
\frac{h_z^{-k}}{h_{z}^n}
\left(\int_{S_z} \omega \right)^{1/p}
\left(\int_{S_z} \omega^{-p'/p} \right)^{1/p'}
\lesssim h_z^{-k},
\]
which, given the definition of $S_T$, the shape regularity of
$\T$, and the finite overlapping property of stars
imply \eqref{stab_PI}.
\end{proof}

\subsection{Interpolation error estimates on simplicial discretizations}
\label{subsec:intP}

The quasi-interpolant 
operator $\Pi_{\T}$ is invariant over the space of polynomials of
degree $m$ on simplicial meshes: 
$\Pi_\T v|_{S_z}  = v$ for $v \in \mathbb{P}_m(S_z)$ and $z \in
\Nin(\T)$ such that $\partial S_z \cap \partial \Omega = \emptyset$. Consequently,
\begin{equation}
 \label{invariance}
\Pi_{\T} Q^m_z \phi =  Q^m_z \phi.
\quad \forall \, \phi\in L^1(\omega,S_z).
\end{equation}
This property, together with \eqref{CO:k_approximation}, yields optimal interpolation estimates for $\Pi_{\T}$.

\begin{theorem}[interpolation estimate on interior simplices]
\label{thm:interpolintsimplex}
Given $T \in \T$ 
such that $\partial T \cap \partial \Omega = \emptyset$ 
and $v \in W_p^{m+1}(\omega,S_T)$, we have the following interpolation error estimate
\begin{equation}
\label{v-Pikp}
  |v - \Pi_{\T} v |_{W^k_p(\omega,T)}
  \lesssim h_{T}^{m + 1 -k } | v |_{W_p^{m+1}(\omega,S_T)},
\quad k = 0,1,\dots,m+1,
  \end{equation}
where the hidden constant depends only on $C_{p,\omega}$, $\sigma$,
$\psi$ and $m$.
\end{theorem}
\begin{proof}
Given $T \in \T$, choose a node $z \in \Nin(T)$.
Property \eqref{invariance} yields, 
\[
  |v - \Pi_{\T} v |_{W^k_p(\omega,T)} \leq
   |v - Q_{z}^m v |_{W^k_p(\omega,T)} + 
  | \Pi_{\T} ( Q_{z}^m v - v ) |_{W^k_p(\omega,T)}.
\]
Combining the stability of $\Pi_{\T}$ given by \eqref{stab_PI} 
together with \eqref{v-Q_mWkp} implies
\[
  |v - \Pi_{\T} v |_{W^k_p(\omega,T)} 
\lesssim
  \sum_{l=0}^k h_T^{l-k}|v - Q_{z}^m v |_{W^l_p(\omega,S_T)}
\lesssim
  h_T^{m+1-k}|v |_{W^{m+1}_p(\omega,S_T)},
\]
which is exactly \eqref{v-Pikp}.
\end{proof}

By using the fact that, $v \in W^{m+1}_p(\omega, \Omega) \cap \Wp(\omega,\Omega)$ implies $\Pi_\T v_{|\partial\Omega} = 0$
we can extend the results of Theorem~\ref{thm:interpolintsimplex} to boundary elements. The proof is an adaption of standard techniques and, 
in order to deal with the weight, those of the aforementioned Theorem~\ref{thm:interpolintsimplex}. 
See also Theorem~\ref{TH:v - PivH1boundary} below.

\begin{theorem}[interpolation estimates on Dirichlet simplices]
Let $ v \in \Wp(\omega,\Omega) \cap W^{m+1}_p(\omega, \Omega)$. If
$T \in \T$ is a boundary simplex, then \eqref{v-Pikp} holds with a 
constant that depends only on $C_{p,\omega}$, $\sigma$ and $\psi$.
\end{theorem}

We are now in the position to write a global interpolation estimate. 

\begin{theorem}[global interpolation estimate over simplicial meshes]
\label{TH:v-Pikp}
Given $\T \in \Tr$ and $v \in W_p^{m+1}(\omega,\Omega)$,
we have the following global interpolation error estimate
\begin{equation}
\label{v-Pikpglob}
\left( \sum_{T\in\T} h_T^{-(m+1-k)p}|v - \Pi_{\T} v |_{W^k_p(\omega,T)}^{p} \right)^{1/p}
  \lesssim  | v |_{W_p^{m+1}(\omega,\Omega)},
\end{equation}
for $k = 0,\ldots,m+1$, where the hidden constant depends only on $C_{p,\omega}$, $\sigma$,
$\psi$ and $m$.
\end{theorem}
\begin{proof}
 Raise \eqref{v-Pikp} to the $p$-th power and add over all $T \in \T$. 
The finite overlapping property of stars of $\T$ yields the result.
\end{proof}

\subsection{Anisotropic interpolation estimates on rectangular meshes}
\label{subsec:intQ}

Narrow or anisotropic elements are those with disparate sizes in 
each direction. They are necessary,
for instance, for the optimal approximation of functions with a 
strong directional-dependent behavior such as line and edge
singularities, boundary layers, and shocks (see \cite{DL:05,DLP:12,NOS}).

Inspired by \cite{DL:05}, here we derive interpolation error estimates assuming only 
that neighboring elements have comparable sizes, thus obtaining results which are valid for a rather general
family of anisotropic meshes. Since symmetry is essential, 
we assume that $\Omega = (0,1)^n$, or that $\Omega$ is any 
domain which can be decomposed into $n$-rectangles.
We use below the notation introduced in \cite{DL:05}.

We assume that the mesh $\T$ is composed of rectangular elements $R$, with sides parallel to the coordinate axes. 
By $\vero \in \N(\T)$ we denote a node or vertex of the triangulation $\T$ and by
$S_{\vero}$, $S_R$ the associated patches; see \S~\ref{sub:sec:preliminars}.
Given $R \in \T$,
we define $h_{R}^i$ as the length of $R$ in the $i$-th direction and, if $\vero \in \N(\T)$, we define
$h_{\vero}^i = \min\{ h_{R}^i: \vero \in R\}$ for $i=1,\cdots,n$. 
The finite element space is defined by \eqref{eq:defFESpace} with $\mathcal{P} = \Q_1$.

We assume the following weak shape regularity condition: there exists a constant $\sigma > 1$, such that
if $R,S \in \T$ are neighboring elements, we have
\begin{equation}
 \label{shape_reg_weak}
 \frac{h_R^i}{h_S^i} \leq \sigma, \qquad i=1,\dots,n.
\end{equation}
Whenever $\vero$ is a vertex of $R$ the shape regularity assumption \eqref{shape_reg_weak} implies that
$h_{\vero}^i$ and $h_{R}^i$ are equivalent up to a constant that depends only on $\sigma$. We define
\[
  \psi_{\vero}(x) =  \frac{1}{ h_{\vero}^1 \dots h_{\vero}^n } \psi
      \left(\frac{\vero_1-x_1}{h_{\vero}^1}, \dots, \frac{\vero_n-x_n}{h_{\vero}^n}\right),
\]
which, owing to \eqref{shape_reg_weak} and $r \leq 1/\sigma$, satisfies $\supp \psi_{\vero} \subset S_{\vero}$. 
Notice that this function incorporates a different length scale on each direction $x_i$, which
will prove useful in the study of anisotropic estimates. 

Given $\omega \in A_p(\R^n)$, and $v \in L^p(\omega,\Omega)$, we define $Q^1_{\vero} v$, the first degree regularized Taylor 
polynomial of $v$ about the vertex $\vero$ as in \eqref{averaged_taylor}. We also define
the quasi-interpolation operator $\Pi_{\T}$ as in
\eqref{def:PI}, \ie upon denoting by $\lambda_{\vero}$ the Lagrange
nodal basis function of $\V(\T)$, $\Pi_{\T} v$ reads
\begin{equation}
\label{Pi_lambda}
  \Pi_{\T} v := \sum_{\vero \in \Ninn(\T)} Q^1_{\vero}v(\vero) \lambda_{\vero}.
\end{equation}

The finite element space $\V(\T)$ is not invariant under the operator defined in \eqref{Pi_lambda}. Consequently, we cannot use the techniques for 
simplicial meshes developed in \S~\ref{subsec:intP}. This, as the results below show, is not a limitation to obtain interpolation error estimates. 

\begin{lemma}[anisotropic $L^p$-weighted error estimates I]
\label{LM:aniso_approximation}
Let $\verot \in \Nin(\T)$. If $v \in W^1_p(\omega,S_{\verot})$, then we have
\begin{equation}
\label{aniso_v-Q_vLp}
  \|v - Q^0_{\verot} v \|_{L^p(\omega,S_{\verot})}
  \lesssim    \sum_{i=1}^n h_{\verot}^i
  \| \partial_{x_i} v\|_{L^p(\omega,S_{\verot})}.
\end{equation}
If $v \in W^2_p(\omega,S_{\verot})$ instead, then the following estimate holds
\begin{equation}
\label{aniso_v-v_vW1p}
  \|\partial_{x_j}(v - Q^1_{\verot}v) \|_{L^p(\omega,S_{\verot})}
  \lesssim  \sum_{i=1}^n h_{\verot}^i
        \| \partial_{x_i}\partial_{x_j} v\|_{L^p(\omega,S_{\verot})},
  \end{equation}
for $j=1,\dots,n$. In both inequalities, the hidden constants depend only on $C_{p,\omega}$, $\sigma$ and $\psi$.
\end{lemma}
\begin{proof}
 To exploit the symmetry of the elements we define the map
\begin{equation}
\label{aniso_map}
\Fm_\vero: x \mapsto \bar{x},
\qquad 
\bar{x}_i = \frac{\vero_i-x_i}{h_{\vero}^i}, \qquad i=1,\dots,n,
\end{equation}
and proceed exactly as in the proof of Lemma~\ref{LM:approximation}.
\end{proof}

Lemma~\ref{LM:aniso_approximation}, in conjunction with
the techniques developed in Lemma~\ref{LM:double_approximation} give rise
the second order anisotropic error estimates in the weighted $L^p$-norm.

\begin{lemma}[anisotropic $L^p$-weighted error estimate II]
\label{LM:aniso_approximation_2}
Let $\verot \in \Nin(\T)$. If $v \in W^2_p(\omega,S_{\verot})$, then we have
\begin{equation}
\label{aniso_v-v_vLp_double}
  \|v - Q^1_{\verot}v \|_{L^p(\omega,S_{\verot})}
  \lesssim  \sum_{i,j=1}^n h_{\verot}^i h_{\verot}^j
        \| \partial_{x_i}\partial_{x_j} v\|_{L^p(\omega,S_{\verot})},
  \end{equation}
where the hidden constant in the inequality above depends only on $C_{p,\omega}$, $\sigma$ and $\psi$.
\end{lemma}
\begin{proof}
Recall that, if $R^1_{\vero}(v) = Q^0_{\vero}( Q^1_\vero v -  v)$,
  then we can write
\begin{equation*}
v -  Q^1_{\vero}v = 
(v - Q^1_{\vero}v) - Q^0_{\vero}(v - Q^1_{\vero}v) - R^1_{\vero} (v).
\end{equation*}
Applying estimates \eqref{aniso_v-Q_vLp} and \eqref{aniso_v-v_vW1p}
successively, we see that
\begin{align*}
\|(v - Q^1_{\vero}v) - Q^0_{\vero}(v - Q^1_{\vero}v)\|_{L^p(\omega,S_\vero)} 
& \lesssim
\sum_{i=1}^n h_{\vero}^i
  \| \partial_{x_i} (v - Q^1_{\vero}v)\|_{L^p(\omega,S_{\vero})}
\\
& \lesssim
\sum_{i,j=1}^n h_{\vero}^i h_{\vero}^j
  \| \partial_{x_i}\partial_{x_j} v \|_{L^p(\omega,S_{\vero})}.
\end{align*}
It remanins then to bound $R^1_{\vero}(v)$. We proceed as
in the proof of \eqref{doble_aux_2} in Lemma~\ref{LM:double_approximation}.
The definition \eqref{averaged_taylor} of the 
averaged Taylor polynomial, together
with the cancellation property
$R^1_{\vero}(v) = R^1_{\vero}(v - Q_{\vero}^1 v)$, implies
\begin{equation*}
 \|R^1_{\vero}(v)\|^p_{L^p(\omega,S_\vero)} \lesssim \sum_{i=1}^n (h^i_{\vero})^p 
\| \partial_{x_i}(v - Q^1_{\vero} v)\|^p_{L^p(\omega,S_{\vero})}
\| 1\|^p_{L^p(\omega,S_{\vero})}
\| \psi_{\vero}\|^p_{L^{p'}(\omega^{-p'/p},S_{\vero})} 
\end{equation*}
Combining \eqref{aniso_v-v_vW1p} with the inequality
$
  \| \psi_{\vero}\|_{L^{p'}(\omega^{-\poverp},S_{\vero})} \| 1\|_{L^p(\omega,S_{\vero})} \lesssim 1,
$ 
which follows  from the
the definition of $\psi_{\vero}$ and the definition \eqref{A_pclass} 
of the $A_p$-class, yields
\begin{equation*}
 \|R^1_\vero(v)\|_{L^p(\omega,S_\vero)} \lesssim
\sum_{i,j=1}^n h^i_{\vero} h^j_{\vero} 
\|\partial_{x_i}\partial_{x_j} v\|_{L^p(\omega,S_{\vero})},
\end{equation*}
and leads to the asserted estimate \eqref{aniso_v-v_vLp_double}.
\end{proof}

The anisotropic error estimate \eqref{aniso_v-Q_vLp} together with the weighted $L^p$ stability
of the interpolation operator $\Pi_{\T}$, enables us to obtain anisotropic weighted $L^p$ interpolation
estimates, as shown in the following Theorem.

\begin{theorem}[anisotropic $L^p$-weighted interpolation estimate I]
\label{full_aniso_Lp}
Let $\T$ satisfy \eqref{shape_reg_weak} and $R \in \T$.
If $v \in L^p(\omega,S_{R})$, we have
\begin{equation}
\label{PiLpboundedaniso}
  \|  \Pi_{\T} v \|_{L^p(\omega,R)} \lesssim \| v \|_{L^p(\omega,S_{R})}.
\end{equation}
If, in addition, $w \in W^1_p(\omega,S_{R})$
and $\partial R \cap \partial \Omega = \emptyset$, then
\begin{equation}
\label{v - PivLp}
  \| v - \Pi_{\T} v \|_{L^p(\omega,R)} \lesssim \sum_{i=1}^n h_{R}^i
  \| \partial_{x_i} v\|_{L^p(\omega,S_{R})}.
\end{equation}
The hidden constants in both inequalities depend only on $C_{p,\omega}$, $\sigma$ and $\psi$.
\end{theorem}
\begin{proof}
The local stability \eqref{PiLpboundedaniso}
of $\Pi_{\T}$ follows from Lemma \ref{LM:stab_PI} with $k=0$. 
Let us now prove \eqref{v - PivLp}. Choose a node
$\vero \in \Nin(R)$. Since $Q^0_{\vero}v$ is constant,
and $\partial R \cap \partial \Omega = \emptyset$, $ \Pi_{\T} Q^0_{\vero}v = Q^0_{\vero}v$
over $R$. This, in conjunction with estimate \eqref{PiLpboundedaniso}, allows us 
to write
\[
  \| v -  \Pi_{\T} v \|_{L^p(\omega,R)} =
  \| (I- \Pi_{\T})(v - Q^0_{\vero}v) \|_{L^p(\omega,R)}
  \lesssim \| v - Q^0_{\vero}v \|_{L^p(\omega,S_R)}.
\]
The desired estimate \eqref{v - PivLp} now follows from Corollary~\ref{C:Poincareweighted-2}.
\end{proof}

To prove interpolation error estimates on the first derivatives for interior elements we follow 
\cite[Theorem~2.6]{DL:05} and use the symmetries of a cube, thus handling the anisotropy 
in every direction separately. We start by studying the case of 
\emph{interior elements.}

\begin{figure}[h!]
  \begin{center}
    \includegraphics[scale=0.3]{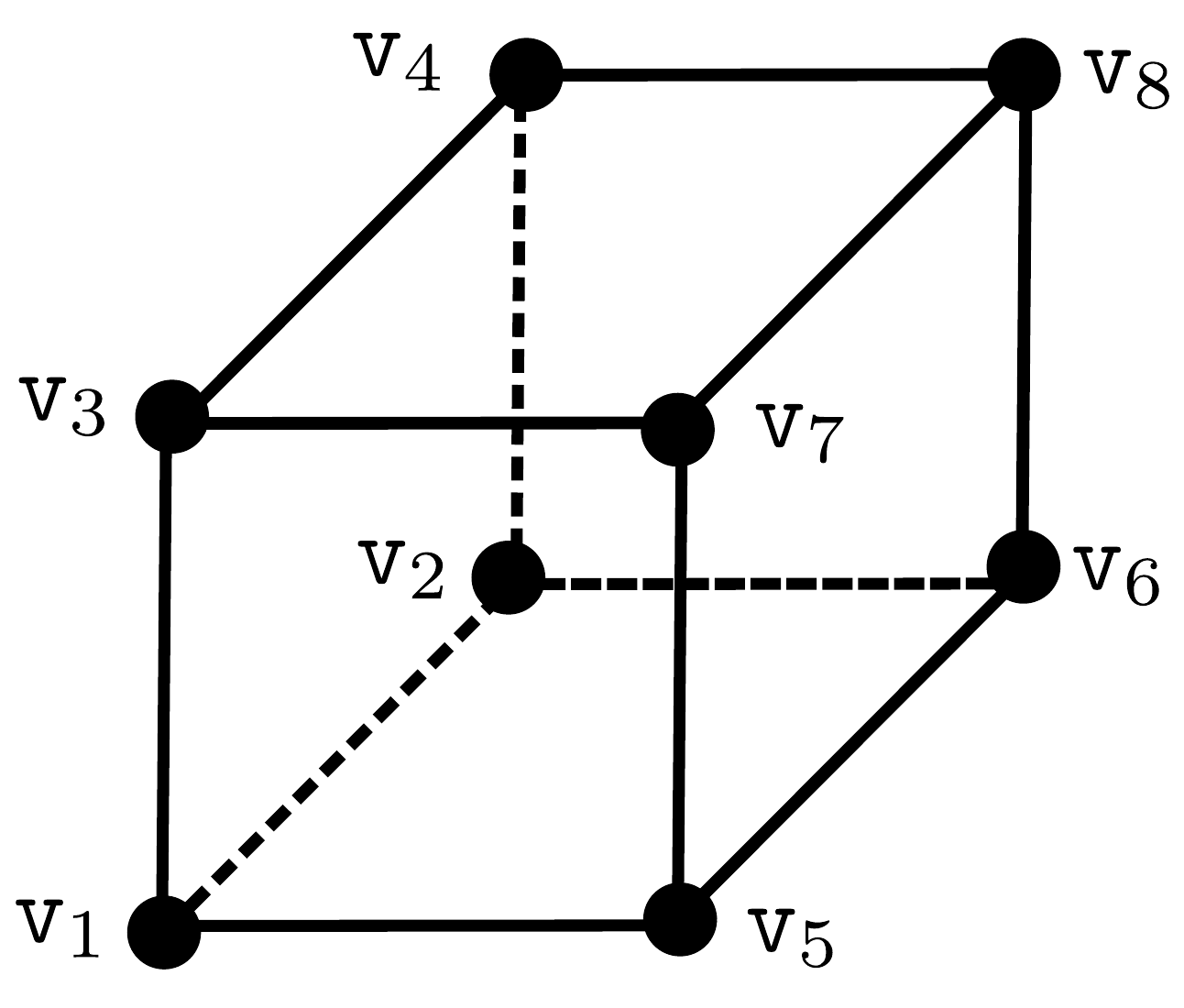}
  \end{center}
  \caption{\small
  An anisotropic cube with sides parallel to the coordinate axes and the labeling of 
  its vertices. The numbering of the vertices proceeds recursively as follows: a cube in dimension
  $m$ is obtained as the Cartesian product of an $(m-1)$-dimensional
  cube with vertices $\{\vero_i\}_{i=1}^{2^{m-1}}$ and an interval, and
  the new vertices are $\{\vero_{i+2^{m-1}}\}_{i=1}^{2^{m-1}}$. }
\label{fig:cube}
\end{figure}

\begin{theorem}[anisotropic $W^1_p$-weighted interpolation estimates]
\label{full_aniso_W1p}
Let
$R \in \T$ be such that $\partial R \cap \partial \Omega = \emptyset$.
If $v \in W^1_p(\omega,S_R)$ we have the stability bound
\begin{equation}
\label{aniso_stab_grad}
 \| \nabla \Pi_\T v \|_{ L^p(\omega,R) } \lesssim \| \nabla v \|_{ [L^p(\omega,S_R) }.
\end{equation}
If, in addition, $v \in W^2_p(\omega,S_R)$ we have, for $j=1,\cdots,n$,
\begin{equation}
\label{aniso_v - PivW1p}
  \| \partial_{x_j} ( v - \Pi_{\T} v )\|_{L^p(\omega,R)} \lesssim  
      \sum_{i=1}^n h_{R}^i  \| \partial_{x_j}\partial_{x_i} v\|_{L^p(\omega,S_R)}.
\end{equation}
The hidden constants in the inequalities above depend only on $C_{p,\omega}$, $\sigma$ and $\psi$.
\end{theorem}
\begin{proof}
Let us bound the derivative with respect to the first argument
$x_1$. The other ones follow from similar considerations.
As in \cite[Theorem~2.5]{DL:05}, to exploit the geometry of $R$, we label its vertices
in an appropriate way: vertices that differ only in the first
component are denoted $\vero_i$ and $\vero_{i+2^{n-1}}$ for 
$i = 1,\dots , 2^{n-1}$; see Figure~\ref{fig:cube} for the three-dimensional case.

Clearly $v - \Pi_{\T} v =(v - Q^1_{\vero_1}v) + (Q^1_{\vero_1}v-\Pi_{\T} v)$,  
and the difference $v - Q^1_{\vero_1}v$
is estimated by Lemma~\ref{LM:aniso_approximation}.
Consequently,
it suffices to consider $q =  Q^1_{\vero_1}v-\Pi_{\T} v \in \Q_1(R)$. 
Thanks to the special labeling of the vertices we have that 
$\partial_{x_1} \lambda_{\vero_{i+2^{n-1}}} = -\partial_{x_1}\lambda_{\vero_{i}}$. Therefore
\[
  \partial_{x_1} q = \sum_{i=1}^{2^{n}} q(\vero_i) \partial_{x_1} \lambda_{\vero_i} =
  \sum_{i=1}^{2^{n-1}} ( q(\vero_i) - q(\vero_{i+2^{n-1}})) \partial_{x_1} \lambda_{\vero_i},
\]
so that
\begin{equation}
\label{partial_x1}
  \| \partial_{x_1} q \|_{L^p(\omega,R)} \leq
  \sum_{i=1}^{2^{n-1}} | q(\vero_i) - q(\vero_{i+2^{n-1}}) |
  \| \partial_{x_1} \lambda_{\vero_i}\|_{ L^p(\omega,R) }.
\end{equation}
This shows that it suffices to estimate $\delta q(\vero_1) = q(\vero_1) - q(\vero_{1+2^{n-1}})$.
The definitions of $\Pi_{\T}$, $q$, and the averaged Taylor polynomial \eqref{averaged_taylor}, imply that
\begin{equation}
\label{deltaq1}
    \delta q(\vero_1)=
  \int P^1v(x,\vero_{1+2^{n-1}}) \psi_{\vero_{1+2^{n-1}}}(x) \diff x 
  - \int P^1 v(x,\vero_{1+2^{n-1}}) \psi_{\vero_{1}}(x) \diff x,
\end{equation}
whence employing
the operation $\circ$ defined in \eqref{eq:defcirc} and changing variables, we get
\begin{multline*}
  \delta q(\vero_1)  = \int \bigg( P^1 v(\vero_{1+2^{n-1}} - h_{\vero_{1+2^{n-1}}} \circ z, \vero_{1+2^{n-1}})
  \\
   - P^1 v(\vero_{1} - h_{\vero_{1}} \circ z,\vero_{1+2^{n-1}}) \bigg)  \psi(z) \diff z.
\end{multline*}
Define 
\[
 \theta_1 = \vero_{1 +2^{n-1}}^1  - \vero_1^1 + (h_{\vero_1}^1 -  h_{\vero_{1 +2^{n-1}} }^1)z_1,
\]
$\theta = (\theta_1,0,\dots,0)$ and, for $t \in [0,1]$, the function
$F_z(t) = P^1 v(\vero_{1} - h_{\vero_{1}} \circ z + t \theta,\vero_{1+2^{n-1}})$.
Since, for $i=2,\cdots,n$ we have that $h_{\vero_1}^i = h_{\vero_{1+2^{n-1}}}^i$ and 
$\vero_1^i = \vero_{1+2^{n-1}}^i$, by using the definition of $\theta$ we arrive at 
\[
 P^1 v(\vero_{1+2^{n-1}} - h_{\vero_{1+2^{n-1}}} \circ z, \vero_{1+2^{n-1}})
     - P^1 v(\vero_{1} - h_{\vero_{1}} \circ z,\vero_{1+2^{n-1}}) = F_z(1) - F_z(0),
\]
and consequently
\[
  \delta q(\vero_1)= \int (F_z(1) - F_z(0))\psi(z)\diff z  = 
  \int_{0}^1 \int F_{z}'(t)\psi(z) \diff z \diff t. 
\]
Since $\psi$ is bounded and $B = \supp \psi \subset B(0,1)$, it suffices to bound
the integral
\[
 I(t) = \int_{B} |F'_z(t)| \diff z .
\]
Invoking the definition of
$F_z$, we get
$
 F_z'(t) = \nabla P^1 v(\vero_1 - h_{\vero_1} \circ z + t \theta,\vero_{1+2^{n-1}}) \cdot\theta,
 $
which, together with the definition of the polynomial $P^1 v$ given by \eqref{taylor}, yields
\begin{align*}
I(t) & \lesssim \int_{B} | \partial^2_{x_1} v(\vero_1 - h_{\vero_1} \circ z + t \theta)| \,
|\vero_{1+2^{n-1}}^1 -\vero_1^1 + h^1_{\vero_1} z_1 - t \theta_1| \, |\theta_1|\diff z
\\
&+ \sum_{i=2}^n \int_{B}| \partial^2_{x_ix_1} v(\vero_1 - h_{\vero_1} \circ z + t \theta)|
\, |\vero_{1+2^{n-1}}^i - \vero_{1}^i + h^i_{\vero_1} z_i | \, |\theta_1|\diff z
\end{align*}
Now, using that $|z| \leq 1$, $0\leq t \leq 1$, and
the definition of $\theta$, we easily see that
  $|\theta|=|\theta_1|\lesssim h_{\vero_1}^1$ as well as
$|\vero_{1+2^{n-1}}^1 -\vero_1 + h^1_{\vero_1} z_1 - t \theta_1| \lesssim h^1_{\vero_1}$ and
$|\vero_{1+2^{n-1}}^i - \vero_{1}^i - h^i_{\vero_1} z_i | \lesssim h^i_{\vero_1}$ 
for $i=2,\dots n$, whence
\[
I(t)  \lesssim \sum_{i=1}^n h_{\vero_1}^{1} h_{\vero_1}^{i}
\int_{B} | \partial^2_{x_ix_1} v(\vero_1 - h_{\vero_1} \circ z + t \theta)| \diff z.
\]
Changing variables via $y = \vero_1 - h_{\vero_1} \circ z + t \theta$, we obtain
\[
I(t)  \lesssim \frac{1}{h_{\vero_1}^{2} \dots h_{\vero_1}^{n} }
\sum_{i=1}^n h_{\vero_1}^{i}
\int_{S_R}| \partial^2_{x_ix_1} v(y)| \diff y,
\]
where we have used that the support of $\psi$ is mapped into $S_{\vero_1} \subset S_{R}$. 
H{\"o}lder's inequality implies
\[
I(t)  \lesssim \frac{1}{h_{\vero_1}^{2} \dots h_{\vero_1}^{n} } \| 1 \|_{L^{p'}(\omega^{-p'/p},S_R)}\sum_{i=1}^n h_{\vero_1}^{i}
\| \partial^2_{x_ix_1} v \|_{L^p(\omega,S_R)}, 
\]
which combined with $\| \partial_{x_1} \lambda_{\vero_1}\|_{L^p(\omega,R) } 
\| 1 \|_{L^{p'}(\omega^{-\poverp},S_R)}\lesssim h_{\vero_1}^{2} \dots h_{\vero_1}^{n}$,
because $\omega \in A_p(\R^n)$, gives the following
bound for the first term in \eqref{partial_x1}
\begin{equation*}
\label{qv2aux2}
  \delta q(\vero_1) \| \partial_{x_1} \lambda_{\vero_1}\|_{ L^p(\omega,R) } 
\lesssim 
\sum_{i=1}^n h_{\vero_1}^{i}
\| \partial^2_{x_ix_1} v \|_{L^p(\omega,S_R)}.
\end{equation*}
This readily yields \eqref{aniso_v - PivW1p}.

The estimate \eqref{aniso_stab_grad} follows along the same arguments
as in \cite[Theorem 4.7]{NOS}. 
In fact, by the triangle inequality
\begin{equation}
\label{dv-Piv}
 \| \nabla\Pi_{\T} v\|_{ L^p(\omega,R) } \leq 
  \| \nabla Q_{\vero_1}^1 v\|_{ L^p(\omega,R) } 
  + \| \nabla(Q_{\vero_1}^1 v- \Pi_{\T} v)\|_{ L^p(\omega,R) }.
\end{equation}
The estimate of the first term on the right hand side of \eqref{dv-Piv} begins by noticing that
the definition of $\psi_{\vero_1}$ and the
definition \eqref{A_pclass} of the $A_p$ class imply
\[
  \| \psi_{\vero_1}\|_{L^{p'}(\omega^{-\poverp},S_R)} \| 1\|_{L^p(\omega,S_R)} \lesssim 1.
\] 
This, together with the definition \eqref{averaged_taylor} of regularized Taylor polynomial $Q^1_{\vero_1}v$, 
yields
\begin{align*}
 \|\nabla Q^1_{\vero_1}v\|_{L^p(\omega,R)} &\leq \| \nabla v\|_{L^p(\omega,S_R) }
  \| \psi_{\vero_1} \|_{L^{p'}(\omega^{-\poverp},S_R)} \| 1 \|_{L^{p}(\omega,S_R)} \\
  &\lesssim \| \nabla v\|_{L^p(\omega,S_R)}.
\end{align*}
To estimate the second term of the right hand side of
\eqref{dv-Piv}, we integrate by parts \eqref{deltaq1}, using that 
$\psi_{\vero_i} =0$ on $\partial S_{\vero_i}$ for $i=1,\dots,n$, to get
\begin{equation*}
\label{I_1+I_2}
  \begin{aligned}
    \delta q(\vero_1) & = (n+1)\left(\int v(x) \psi_{\vero_{1+2^{n-1}}}(x)\diff x - 
    \int v(x) \psi_{\vero_{1}}(x)\diff x\right) \\
    & - \int v(x)( \vero_{1+2^{n-1}}- x)\cdot \nabla \psi_{{\vero_{1+2^{n-1}}}}(x)\diff x
    + \int v(x)( \vero_{1}- x) \cdot \nabla \psi_{\vero_1}(x)\diff x  .
  \end{aligned}
\end{equation*}
In contrast to \eqref{deltaq1}, we have now created differences which
involve $v(x)$ instead of $\nabla v(x)$. However,
the same techniques used to derive \eqref{aniso_v - PivW1p} yield
\[
 |\delta q(\vero_1)| \lesssim \frac{1}{h_{\vero_1}^2 \dots h_{\vero_1}^n} 
\| \nabla v \|_{ L^{p}(\omega,S_R) } \| 1\|_{L^{p'}(\omega^{-p'/p},S_R)},
\]
which, since 
$\| \partial_{x_1} \lambda_{\vero_1}
\|_{L^{p'}(\omega^{-\poverp},S_R)} \| 1\|_{L^p(\omega,S_R)} \lesssim
h_{\vero_1}^{2} \dots h_{\vero_1}^{n}$, 
results in
\begin{equation*}
 | \delta q(\vero_1)|\| \partial_{x_1} \lambda_{\vero_1}\|_{ L^p(\omega,R) }  \lesssim \| \nabla v \|_{L^{p}(\omega,S_R)]^n }. 
\end{equation*}
Replacing this estimate in \eqref{partial_x1}, we get 
\[
\| \nabla( Q^1_{\vero_1} v- \Pi_{\T} v)\|_{ L^p(\omega,R)]^n } \lesssim \| \nabla v\|_{ L^p(\omega,S_R) },
\]
which implies the desired result \eqref{aniso_stab_grad}.
This completes the proof.
\end{proof}

Let us now derive a second order anisotropic interpolation error estimates for
the weighted $L^p$-norm, which is novel even for unweighted norms.
For the sake of simplicity,
and because the arguments involved are rather technical (as in
Theorem~\ref{full_aniso_W1p}), we 
prove the result in two dimensions.
However, analogous results can be obtained in three
and more dimensions by using similar arguments.

\begin{theorem}[anisotropic $L^p$-weighted interpolation estimate II]
\label{full_aniso_Lp_II}
Let $\T$ satisfy \eqref{shape_reg_weak} and $R \in \T$
such that $\partial R \cap \partial \Omega = \emptyset$.
If $v \in W^2_p(\omega,S_{R})$, then we have
\begin{equation}
\label{double:v - PivLp}
  \| v - \Pi_{\T} v \|_{L^p(\omega,R)} \lesssim \sum_{i,j=1}^n h_{R}^ih_{R}^j
  \| \partial_{x_i} \partial_{x_j} v\|_{L^p(\omega,S_{R})},
\end{equation}
where the hidden constant in the inequality above depends only on $C_{p,\omega}$, $\sigma$ and $\psi$.
\end{theorem}
\begin{proof}
To exploit the symmetry of $R$, 
we label its vertices of $R$ according to Figure~\ref{fig:cube}:
$\vero_2=\vero_1+(a,0), \vero_3=\vero_1+(0,b),\vero_4=\vero_1+(a,b)$.
We write $v - \Pi_{\T} v = (v - Q^1_{\vero_1} v) + (Q^1_{\vero_1} v - \Pi_{\T} v)$.
The difference $v - Q^1_{\vero_1} v$ is estimated by Lemma
\ref{LM:aniso_approximation_2}. Consequently,
it suffices to estimate $q = Q^1_{\vero_1} v - \Pi_{\T} v$.

Since $q \in \V(\T)$,
\begin{equation}
\label{q}
 q = \sum_{i=1}^4 q(\vero_i) \lambda_{\vero_i} \implies
 \| q \|_{L^p(\omega,R)} \leq \sum_{i=1}^4 |q(\vero_i)|  \| \lambda_{\vero_i} \|_{L^p(\omega,R)},
\end{equation}
and we only need to deal with $q(\vero_i)$ for $i=1,\dots,4$.
Since $q(\vero_1) = 0$, in accordance with the definition \eqref{Pi_lambda} of
$\Pi_{\T}$, we just consider $i =2$.
Again, by \eqref{Pi_lambda}, we have
\begin{align*}
 q(\vero_2) & = Q^1_{\vero_1} v(\vero_2) - Q^1_{\vero_2}v(\vero_2) 
\end{align*}
which, together with the definition of the averaged Taylor 
polynomial \eqref{averaged_taylor} and a change of variables, yields
\begin{align*}
q(\vero_2) & = \int  \left( P^1 v(\vero_1 - h_{\vero_1} \circ z, \vero_2) - 
P^1 v(\vero_2 - h_{\vero_2} \circ z, \vero_2) \right) \psi(z) \diff z. 
\end{align*}
To estimate this integral, we define $\theta = (\theta_1,0)$, where $\theta_1 = \vero_1^1 - \vero_2^1 + 
(h^1_{\vero_2} - h^1_{\vero_1})z_1$, and the function
$F_z(t) = P^1v(\vero_2 - h_{\vero_2} \circ z + t \theta , \vero_2)$. Exploiting
the symmetries of $R$, \ie using that $\vero_1^2 = \vero_2^2$ and $h_{\vero_1}^2 = h_{\vero_2}^2$, we arrive at
\[
 q(\vero_2) = \int \big( F_z(1) - F_z(0) \big) \psi(z) \diff z = 
\int_0^1 \int F_z'(t) \psi(z) \diff z \diff t.
\]
By using the definition of the Taylor polynomial $P^1 v$ given in \eqref{taylor}, we obtain
\[
F_z'(t) = \theta D^2 v (\vero_2 - h_{\vero_2} \circ z + t \theta) ( h_{\vero_2} \circ z - t \theta)
\]
which, together with the definition of $\theta$ and the inequalities
$ | \theta_1 | \lesssim h_{\vero_2}^1$, $|  h^1_{\vero_2} z_1 - t \theta_1  | \lesssim h^1_{\vero_2} $
and $|  h^2_{\vero_2} z_2 | \lesssim h_{\vero_2}^2$, implies
\begin{align*}
 \int F_z'(t) \psi(z) \diff z 
 & \leq \int | \partial_{x_1 x_1}  v(\vero_2 - h_{\vero_2} \circ z + t \theta)|
 \, |  h^1_{\vero_2} z_1 - t \theta_1| \, | \theta_1 | \, | \psi(z) | \diff z
\\
 & + \int | \partial_{x_2 x_1}  v(\vero_2 - h_{\vero_2} \circ z + t \theta)|
 \, | h^2_{\vero_2} z_2 | \, | \theta_1 | \, | \psi(z) | \diff z
 \\
 & \lesssim h_{\vero_2}^1h_{\vero_2}^1 \int | \partial_{x_1 x_1}
 v(\vero_2 - h_{\vero_2} \circ z + t \theta)| \, | \psi(z) | \diff z
\\
 & + h_{\vero_2}^2 h_{\vero_2}^1 \int |\partial_{x_2 x_1}
 v(\vero_2 - h_{\vero_2} \circ z + t \theta)| \, | \psi(z) | \diff z.
\end{align*}
The change of variables $y = \vero_2 - h_{\vero_2} \circ z + t \theta $ yields
\[
  \int F_z'(t) \psi(z) \diff z \lesssim 
  \left( \frac{h_{\vero_2}^1}{h_{\vero_2}^2}\| \partial_{x_1x_1} v\|_{L^p(\omega,S_R)}  +  \| \partial_{x_2x_1} v\|_{L^p(\omega,S_R)} \right) 
  \| 1 \|_{L^{p'}(\omega^{-p'/p},S_R)},
\]
where we used H{\"o}lder inequality, that the support of $\psi$ is mapped into $S_R$, and $\psi \in L^{\infty}(\R^n)$. Finally,
using the $A_p$-condition, we conclude
\[
 |q(\vero_2)| \| \lambda_{\vero_2}\|_{L^p(\omega,R)}
 \lesssim ( h_{\vero_2}^1)^2 \| \partial_{x_1x_1}
   v\|_{L^p(\omega,S_R)} + 
    h_{\vero_2}^1 h_{\vero_2}^2 \| \partial_{x_2x_1} v\|_{L^p(\omega,S_R)}.
\]

The same arguments above apply to the remaining terms in \eqref{q}.
For the term labeled $i=3$, we obtain
\[
 |q(\vero_3)| \| \lambda_{\vero_3}\|_{L^p(\omega,R)}
 \lesssim (h_{\vero_3}^2)^2 \| \partial_{x_2x_2} v\|_{L^p(\omega,S_R)} + h_{\vero_3}^1 h_{\vero_3}^2 \| \partial_{x_1x_2} v\|_{L^p(\omega,S_R)},
\]
whereas for the term labeled $i=4$, rewritten first in the form
\begin{align*}
q(\vero_4) 
 = \big( Q_{\vero_1}^1 v (\vero_4) -  Q_{\vero_3}^1 v (\vero_4) \big)
+ \big( Q_{\vero_3}^1 v (\vero_4) - Q_{\vero_4}^1 v (\vero_4) \big),
\end{align*}
we deduce
\[
 |q(\vero_4)| \| \lambda_{\vero_4}\|_{L^p(\omega,R)}
\lesssim \sum_{i,j=1}^2 h_{\vero_4}^ih_{\vero_4}^j
  \| \partial_{x_i} \partial_{x_j} v\|_{L^p(\omega,S_{R})}.
\]

Finally, replacing the previous estimates back into \eqref{q}, 
and using the shape regularity properties
$h_{\vero_i}^j \approx h_{R}^j $ for $i=1,\dots,4$ and $j=1,2$, which
result from \eqref{shape_reg_weak}, shows the 
desired anisotropic estimate \eqref{double:v - PivLp}.
\end{proof}

Let us comment on the extension of  
the interpolation estimates of Theorem~\ref{full_aniso_W1p} to elements that 
intersect the Dirichlet boundary, where the functions to be approximated vanish. 
The proof is very technical and is an adaptation of the 
arguments of \cite[Theorem 3.1]{DL:05} and \cite[Theorem~4.8]{NOS}, together
with the ideas involved in the proof of Theorem \ref{full_aniso_W1p}
to deal with the Muckenhoupt weight $\omega \in A_p(\R^n)$.

\begin{theorem}[stability and local interpolation: Dirichlet elements]
\label{TH:v - PivH1boundary}
Let $R \in \T$ be a boundary element. If $ v \in W^1_p(\omega,S_R)$ 
and $v = 0$ on $\partial R \cap \partial \Omega$, then we have 
\begin{equation}
\label{dPIcontinuous_boundary}
  \|\nabla \Pi_{\T} v \|_{L^p(\omega,R) } \lesssim 
  \| \nabla v\|_{L^p(\omega,S_R) }.
\end{equation}
Moreover, if $ v \in W^2_p(\omega,S_R)$, then
\begin{equation}
\label{v - PivW1p_boundary}
  \|\partial_{x_j}( v - \Pi_{\T} v) \|_{L^p(\omega,R)} \lesssim 
  \sum_{i=1}^n h_R^i \| \partial_{x_j} \partial_{x_i} v \|_{L^p(\omega,S_R)}. 
\end{equation}
for $j=1,\dots,n$. 
The hidden constants in both inequalities depend only on $C_{p,\omega}$, $\sigma$ and $\psi$.
\end{theorem}

\section{Interpolation estimates for different metrics}
\label{sub:sec:diff_weight}

Given $v \in W^1_p(\omega,S_T)$ with $\omega \in A_{p}(\R^n)$ and $p \in (1,\infty)$,
the goal of this section is to derive local interpolation estimates
for $v$ in the space $L^q(\rho,T)$, with weight $\rho \neq \omega $ 
and Lebesgue exponent $q \ne p$. To derive such 
an estimate, it is necessary to ensure that the function $v$ belongs to $L^q(\rho,T)$, that is we need to
discuss embeddings between weighted Sobolev spaces with different
weights and Lebesgue exponents.

Embedding results in spaces of weakly differentiable functions are fundamental in the analysis of partial 
differential equations. They provide some basic tools in the study of existence, uniqueness and 
regularity of solutions. To the best of our knowledge, the first to prove such a result
was S.L.~Sobolev in 1938 \cite{Sobolev:38}. Since then, a great deal
of effort has been devoted to studying 
and improving such inequalities; see, for instance, \cite{MR1450401,MR0374877,MR1014685}.
In the context of weighted Sobolev spaces, there is an abundant literature that studies the dependence
of this result on the properties of the weight; 
see \cite{FGW:94,GU,GO:88,H:96,HK:00,HS:08,Haroske}.

Let us first recall the embedding results in the classical case,
which will help us draw an analogy 
for the weighted case. We recall the \emph{Sobolev number} of $W^m_p(\Omega)$
\[
 \Sob(W^m_p) = m - \frac{n}{p},
\]
which governs the scaling properties of the seminorm $| v |_{W^m_p(\Omega)}$:
the change of variables $\hat{x} = x/h$ transforms 
$\Omega$ into $\hat\Omega$ and $v$ into $\hat v$, while 
the seminorms scale as
\[
 |\hat v |_{W^m_p(\hat \Omega)} = 
h^{\Sob({W^m_p})} | v|_{W^m_p(\Omega)}.
\]
With this notation classical embeddings \cite[Theorem 7.26]{GT} can be
written in a concise way: if
$\Omega$ denotes an open and bounded domain with Lipschitz boundary,
$1 \leq p < n$ and $\Sob(W^1_p) \geq \Sob(L^q)$, then 
$\Wonepo(\Omega) \hookrightarrow L^{q}(\Omega)$ and
\begin{equation}
\label{unweighted_embedding}
\| v\|_{L^q(\Omega)} \lesssim 
\diam(\Omega)^{\Sob(W^1_p) - \Sob(L^q) }
 \| \nabla v \|_{L^p(\Omega)}
\end{equation}
for all $v \in \Wonepo(\Omega)$. When $\Sob(W^1_p) > \Sob(L^q)$ the embedding is compact.
Results analogous to \eqref{unweighted_embedding} in the weighted setting have been studied
in \cite{CW:85,FGW:94,MR2454455,MR1052009} for $n>1$.
For $n=1$, if $\Omega = (0,a)$, $v \in W^1_p(\omega,\Omega)$, and 
$\omega \in A_p(\R^n)$, Proposition \ref{pro:loc_int} yields $v \in W_1^1(\Omega)$.
Consequently $v \in L^{\infty}(\Omega)$, and then $v \in L^q(\rho,\Omega)$ for any weight $\rho$
and $q \in (1,\infty)$. However, to gain intuition on the explicit
dependence of the embbedding constant in terms of the weights and the 
Lebesgue measure of the domain, let us consider the trivial case 
$n=1$ in more detail.
To simplify the discussion assume that $v(0) = v(a) = 0$. We thus have
\begin{align*}
 \int_{0}^{a} | v(x) |^q \rho(x) \diff x & = \int_0^{a} \rho(x) \left| \int_0^x v'(s) \omega(s)^{1/p} \omega(s)^{-1/p} \diff s\right|^q \diff x 
\\
& \leq \int_0^{a} \rho(x) \left( \int_0^x \omega(s) |v'(s)|^p \diff s \right)^{q/p}
\left( 
\int_0^x \omega(s)^{-\poverp } \diff s \right)^{q/p'} 
\diff x
\end{align*}
whence invoking the definition of the the Muckenhoupt class
\eqref{A_pclass} we realize that
\begin{equation*}
\int_{0}^{a} | v(x) |^q \rho(x) \diff x
\lesssim \| v' \|^q_{L^p(\omega,\Omega)}  
|\Omega|^{q} \rho(\Omega)\omega(\Omega)^{-q/p}.
\end{equation*}
The extension of this result to the $n$-dimensional
case has been studied in \cite{CW:85,FGW:94,MR2454455} and is reported
in the next two theorems; see \cite{CW:85} for a discussion.

\begin{theorem}[embeddings in weighted spaces]
\label{thm:embeddingdiffmetricsI}
Let $\omega \in A_p(\R^n)$, $p \in (1,q]$,
and $\rho$ be a weight that satisfies the strong doubling property
\eqref{strong_double}. Let the pair $(\rho,\omega)$ satisfy the
compatibility condition
\begin{equation}
\label{weight_scaling}
 \frac{r}{R}\left( \frac{\rho(B(x,r))}{\rho(B(x,R))} \right)^{1/q} \leq C_{\rho,\omega}\left( 
 \frac{\omega(B(x,r))}{\omega(B(x,R))}\right)^{1/p},
\end{equation}
for all $x \in \Omega$ and $r \leq R$. If $v \in \Wonepo(\omega,\Omega)$, then
$v\in L^q(\rho,\Omega)$ and
\begin{equation}
\label{embeddingnI}
 \|v\|_{L^q(\rho,\Omega)}  \lesssim \diam(\Omega) 
 \rho(\Omega)^{1/q} 
\omega(\Omega)^{-1/p} 
\| \nabla v \|_{ L^p(\omega,\Omega) },
\end{equation}
where the hidden constant depends on the quotient between 
the radii of the balls inscribed and circumscribed in $\Omega$.
\end{theorem}
\begin{proof}
Given $v \in \Wonepo(\omega, \Omega)$ we denote by $\tilde{v}$ its extension
by zero to a ball $B_R$ of radius $R$ containing $\Omega$
such that $R \leq 2 \diam(\Omega)$.
We then apply \cite[Theorem 1.5]{CW:85} if $p<q$, or \cite[Corollary 2.1]{MR2454455} if $p=q$, to conclude
\begin{align*}
 \| \tilde{v}\|_{L^q(\rho,B_R)} 
  & \lesssim R \rho(B_R)^{1/q} \omega(B_R)^{-1/p}
    \| \nabla \tilde{v} \|_{ L^p(\omega,B_R) }.
\end{align*}
By assumption $\rho$ satisfies the strong doubling property \eqref{strong_double} and so, for $B_r \subset \Omega \subset \bar{\Omega} \subset B_R$, we have
$\rho (B_R) \lesssim \rho(B_r) \leq \rho (\Omega)$ with a constant
that only depends on $R/r$. 
Applying this property, together with
$\omega(\Omega) \le \omega(B_R)$, we derive \eqref{embeddingnI}.
\end{proof}

\begin{theorem}[Poincar\'e inequality]
\label{thm:embeddingdiffmetrics}
Let $p \in (1,q]$, $\rho$ be a weight that satisfies the strong doubling property \eqref{strong_double},
and $\omega \in A_p(\R^n)$, and let the pair $(\rho, \omega)$   satisfy \eqref{weight_scaling}. If
$v \in W^1_p(\omega, \Omega)$, then
there is a constant $v_\Omega$ such that
\begin{equation}
\label{embeddingn}
 \|v - v_\Omega \|_{L^q(\rho,\Omega)}  \lesssim \diam(\Omega) 
 \rho(\Omega)^{1/q} 
\omega(\Omega)^{-1/p} 
\| \nabla v \|_{ L^p(\omega,\Omega) },
\end{equation}
where the hidden constant depends on the quotient between 
the radii of the balls inscribed and circumscribed in $\Omega$.
\end{theorem}
\begin{proof}
Since $\Omega$ is open and bounded, we can choose $0<r<R$ such that
$\bar{B}_r \subset \Omega \subset \bar{\Omega} \subset B_R$,
where $B_\delta$ is a ball of radius $\delta$.
The extension theorem on weighted Sobolev spaces proved in
\cite[Theorem 1.1]{Chau:92} shows that there exists
$ \tilde{v} \in W^1_p(\omega, B_R)$ such that $\tilde{v}_{|\Omega} = v$ and
\begin{equation}
\label{eq:cont-extension}
  \| \nabla \tilde{v} \|_{L^p(\omega,B_R)} \lesssim \| \nabla v \|_{ L^p(\omega,\Omega) },
\end{equation}
where the hidden constant does not depend on $v$. If $p<q$, 
then we invoke \cite[Theorem 1]{FGW:94} and 
\cite[Theorem 1.3]{CW:85} to show that inequality \eqref{embeddingn}
holds over $B_R$ with $v_\Omega$ being a weighted mean of
$\tilde v$ in $B_R$. If $p=q$ instead, we appeal to 
\cite[Remark 2.3]{MR2454455} and arrive at the same conclusion. Consequently, we have
\begin{align*}
  \| \tilde{v} - v_\Omega \|_{L^q(\rho,\Omega)} & \leq  
    \|\tilde{v} - v_\Omega \|_{L^q(\rho,B_R)}
    \lesssim R \rho(B_R)^{1/q} \omega(B_R)^{-1/p} \| \nabla \tilde{v} \|_{ L^p(\omega,B_R) }.
\end{align*}
The strong doubling property
$\rho(B_R)\lesssim\rho(\Omega)$ and $\omega(\Omega)\le\omega(B_R)$
yield
\begin{equation*}
  \| \tilde{v} - v_\Omega \|_{L^q(\rho,\Omega)} \lesssim
    \diam(\Omega) \rho(\Omega)^{1/q} \omega(\Omega)^{-1/p}
    \| \nabla \tilde{v} \|_{L^p(\omega,B_R) }.
\end{equation*}
Employing \eqref{eq:cont-extension} we finally conclude \eqref{embeddingn}.
\end{proof}

Inequalities \eqref{embeddingnI} and \eqref{embeddingn} are
generalizations of several classical results. We first consider
$\omega = \rho \equiv 1$, for which an easy manipulation shows that \eqref{weight_scaling} holds if 
$\Sob(W^1_p) \geq \Sob(L^q)$, whence \eqref{embeddingn} reduces to 
\eqref{unweighted_embedding}. We next consider $\rho = \omega \in A_p(\R^n)$, for which \eqref{weight_scaling} becomes
\[
 \omega(B(x,R)) \lesssim \left( \frac{R}{r}\right)^{pq/(q-p)} \omega(B(x,r)).
\]
This is a consequence of the strong doubling property 
\eqref{strong_double} for $\omega$
in conjunction with $|B_R| \approx R^n$, provided 
the restriction $q \leq pn/(n-1)$ between $q$ and $p$ is valid.
Moreover, owing to the so-called \emph{open ended property}
  of the Muckenhoupt classes \cite{Muckenhoupt}: if $\omega \in
  A_p(\R^n)$, then
$\omega \in A_{p-\epsilon}(\R^n)$ for some $\epsilon > 0$, we conclude that $q \leq pn/(n-1) + \delta$ for some $\delta >0$,
thus recovering the embedding results proved by Fabes, Kenig and Serapioni 
\cite[Theorem 1.3]{FKS:82} and \cite[Theorem 1.5]{FKS:82}; 
see \cite{CW:85} for details.

The embedding result of Theorem~\ref{thm:embeddingdiffmetrics} allows us to obtain
polynomial interpolation error estimates in $L^q(\rho,T)$ for functions in $W^1_p(\omega,S_T)$.

\begin{theorem}[interpolation estimates for different metrics I]
\label{TH:v - PivLq}
Let $\T$ be a simplicial mesh and $\mathcal{P} = \mathbb{P}_1$ in  \eqref{eq:defFESpace}.
Let the pair $(\rho,\omega) \in A_{q}(\R^n)\times A_p(\R^n)$
 satisfy \eqref{weight_scaling}.
If $v \in W^1_p(\omega,S_{T})$ for any $T \in \T$, then
\begin{equation}
\label{eq:v - PivLq}
  \| v - \Pi_{\T} v \|_{L^q(\rho,T)} \lesssim h_{T} \rho(S_T)^{1/q} \omega(S_T)^{-1/p}
      \| \nabla v\|_{ L^p(\omega,S_{T})},
\end{equation}
where the hidden constant depends only on $\sigma$, $\psi$, $C_{p,\omega}$ and $C_{\rho,\omega}$.
\end{theorem}
\begin{proof}
Given an interior element $T \in \T$, let us denote $v_T$ the constant such that
the estimate \eqref{embeddingn} holds true on $S_T$. Since $v_T$ is
constant over ${S_T}$, we have that 
$ \Pi_{\T} v_T = v_T$ in $T$.
This, together with the stability bound \eqref{stab_PI} for the
operator $\Pi_{\T}$, implies
\[
  \| v -  \Pi_{\T} v \|_{L^q(\rho,T)} =
  \| (I- \Pi_{\T})(v - v_T) \|_{L^q(\rho,T)}
  \lesssim \| v - v_T \|_{L^q(\rho,S_T)}.
\]
The Poincar\'e inequality \eqref{embeddingn}
and the mesh regularity assumption \eqref{shape_reg_weak} yield
\[
   \| v -  \Pi_{\T} v \|_{L^q(\rho,T)} \lesssim \| v - v_T \|_{L^q(\rho,S_T)}
   \lesssim h_{T}  \rho(S_T)^{1/q} \omega(S_T)^{-1/p} \| \nabla v \|_{ L^p(\omega,S_T) }
\]
which is \eqref{eq:v - PivLq}.
A similar argument yields \eqref{eq:v - PivLq} on boundary elements.
\end{proof}

A trivial but important consequence of Theorem~\ref{TH:v - PivLq} is the standard, unweighted,
interpolation error estimate in Sobolev spaces; see \cite[Theorem 3.1.5]{Ciarletbook}.

\begin{corollary}[$L^q$-based interpolation estimate]
\label{CR:v - PivLq}
If $p<n$ and $\Sob(W^1_p)>\Sob(L^q)$, then
for all $T \in \T$ 
and $v \in W^1_p(S_{T})$, we have the local error estimate
\begin{equation}
\label{v - PivLq}
  \| v - \Pi_{\T} v \|_{L^q(T)} \lesssim
  h_{T}^{\Sob(W^1_p) - \Sob(L^q) }
  \| \nabla v\|_{ L^p(S_{T}) },
\end{equation}
where the hidden constant depends only on $\sigma$ and $\psi$.
\end{corollary}

For simplicial meshes, the invariance property of $\Pi_\T$ 
and similar arguments to those used in \S~\ref{subsec:intP}
enable us to obtain other interpolation estimates. We illustrate this in the following.

\begin{theorem}[interpolation estimates for different metrics II]
\label{thm:W1qrhow2pomega}
Let $\T$ be a simplicial mesh and $\mathcal P = \mathbb{P}_1$ in
\eqref{eq:defFESpace}. Given $p\in(1,q]$, let the pair 
$(\omega,\rho) \in A_p(\R^n) \times A_q(\R^n)$ satisfy \eqref{weight_scaling}. Then, for every $T \in \T$ 
and every $v \in W^2_p(\omega,S_T)$ we have
\begin{equation}
  \| \nabla (v - \Pi_\T v ) \|_{L^q(\rho,T) } \lesssim h_T \rho(S_T)^{1/q} \omega(S_T)^{-1/p} | v |_{W^2_p(\omega,S_T)},
\label{eq:diffweightsH1H2}
\end{equation}
where the hidden constant depends only on $\sigma$, $\psi$, $C_{p,\omega}$ and $C_{\rho,\omega}$. 
\end{theorem}
\begin{proof}
Let, again, $T \in \T$ be an interior element, the proof for boundary elements follows from similar arguments. Denote by 
$\vero$ a vertex of $T$.
Since the pair of weights $(\omega,\rho)$ satisfies \eqref{weight_scaling} the embedding $W^2_p(\omega,S_T) \hookrightarrow W^1_q(\rho,S_T)$ holds and it is
legitimate to write
\begin{align*}
  \| \nabla (v - \Pi_\T v ) \|_{ L^q(\rho,T) } \leq \| \nabla v - \nabla Q^1_\vero v  \|_{L^q(\rho,T)} 
  + \| \nabla (Q^1_\vero v - \Pi_\T v) \|_{ L^q(\rho,T) }
\end{align*}
In view of \eqref{invariance} and \eqref{stab_PI}, we have
\[
  \| \nabla (Q^1_\vero v - \Pi_\T v) \|_{L^q(\rho,T)} \lesssim \|\nabla  v - \nabla Q^1_\vero v  \|_{L^q(\rho,T)}.
\]
We now recall \eqref{prop_der}, namely $\nabla Q^1_\vero v = Q^0_\vero \nabla v$, to
end up with
\begin{equation*}
  \| \nabla (v - \Pi_\T v ) \|_{L^q(\rho,T)} \lesssim \| \nabla v - Q^0_\vero \nabla v  \|_{L^q(\rho,T)} 
  \lesssim \| \nabla v - (\nabla v)_T  \|_{L^q(\rho,T)},
\end{equation*}
because $Q^0_\vero c = c$ for any constant $c$ and $Q_{\vero}^0$ is continuous in $L^q(\rho,T)$. 
Applying \eqref{embeddingn} finally implies \eqref{eq:diffweightsH1H2}.
\end{proof}

\section{Applications}
\label{sec:applications}

We now present some immediate applications of the interpolation error
estimates developed in the previous sections. We recall that
$\V(\T)$ denotes the finite element space over the mesh $\T$,
$\Pi_{\T}$ the quasi-interpolation operator defined in \eqref{def:PI},
and $U_{\T}$ the Galerkin solution to \eqref{weighted_second_discrete}.

\subsection{Nonuniformly elliptic boundary value problems}
\label{sub:degeneratePDE}

We first derive novel error estimates for the finite element approximation
of solutions of a \textit{nonuniformly} elliptic boundary value problem. 
Let $\Omega$ be a polyhedral domain in $\R^n$
with Lipschitz boundary, $\omega \in A_2(\R^n)$ and $f$ be a function in $L^2({\omega^{-1}},\Omega)$.
Consider problem \eqref{weighted_second} with $\A$ as in \eqref{eq:weightissingular}.
The natural space to seek a solution $u$ of problem
\eqref{weighted_second} is the weighted Sobolev space $H_0^1(\omega,\Omega)$.

Since $\Omega$ is bounded and $\omega \in A_2(\R^n)$, Proposition~\ref{PR:banach} shows that
$H^1_0(\omega,\Omega)$ is Hilbert. The Poincar\'e inequality proved in \cite[Theorem 1.3]{FKS:82}
and the Lax-Milgram lemma then imply the existence and uniqueness of a solution to \eqref{weighted_second}
as well as \eqref{weighted_second_discrete}.
The following result establishes a connection between $u$
and $U_{\T}$.

\begin{corollary}[error estimates for nonuniformly elliptic PDE]
\label{col:errest}
Let $\omega \in A_2(\R^n)$ and
$\V(\T)$ consist of simplicial elements of degree $m\ge 1$ or rectangular elements of degree $m=1$.
If the solution $u$ of \eqref{weighted_second} satisfies
$u \in  H_0^1(\omega,\Omega) \cap H^{k+1}(\omega,\Omega)$ for some $1\le k \le m$,
then we have the following global error estimate
\begin{equation}
\label{estimate_co_1}
 \|\nabla( u - U_{\T}) \|_{ L^2(\omega,\Omega) } \lesssim
\| h^k D^{k+1} u \|_{L^2(\omega,\Omega)},
\end{equation}
where $h$ denotes the local mesh-size function of $\T$.
\end{corollary}
\begin{proof}
By Galerkin orthogonality we have
\[
 \|\nabla( u - U_{\T}) \|_{L^2(\omega,\Omega)} \lesssim
 \inf_{V \in \V(\T)} \| \nabla(u - V) \|_{L^2(\omega,\Omega)}.
\]
Consider $V = \Pi_{\T}u$ and use the local estimates
of either Theorem~\ref{TH:v-Pikp} or Theorems~\ref{full_aniso_W1p}
and \ref{TH:v - PivH1boundary}, depending on the discretization.
This concludes the proof.
\end{proof}

\begin{remark}[regularity assumption]
We assumed that $u \in
H^{m+1}(\omega,\Omega)$ in Corollary~\ref{col:errest}.
Since the coefficient matrix $\A$ is not smooth
but rather satisfies \eqref{eq:weightissingular}, it is natural
to ponder whether $u \in H^{m+1}(\omega,\Omega)$ holds. 
References \cite{MR2780884,MR2024992} provide sufficient conditions on 
$\A$, $\Omega$ and $f$ for this result to be true for $m=1$.
\end{remark}

\begin{remark}[multilevel methods]
Multilevel methods are known to exhibit linear complexity for the
solution of the ensuing algebraic systems. We refer to \cite{GSS:07} for
weights of class $A_1$ and \cite{CNOS:14} for weights of class $A_2$
(including fractional diffusion).
\end{remark}

\subsection{Elliptic problems with Dirac sources}
\label{sub:delta}

Dirac sources arise in applications
as diverse as modeling of pollutant transport, degradation in an aquatic medium
\cite{ABR-CMAME} and problems in fractured domains \cite{DAngelo:SINUM}.
The analysis of the finite element method applied to such problems is not standard, since
in general the solution does not belong to $H^1(\Omega)$ for $n \geq 1$.
A priori error estimates in the $L^2(\Omega)$-norm
have been derived in the literature
using different techniques.
In a two dimensional setting and assuming that the domain is smooth,
Babu{\v{s}}ka \cite{Babu:73} derived almost
optimal a priori error estimates of order
$\mathcal{O}(h_{}^{1-\epsilon})$, for an arbitrary $\epsilon >0$.
Scott \cite{Scott:73} improved these estimates by removing the
$\epsilon$ and thus obtaining an optimal
error estimate of order $\mathcal{O}(h_{}^{2-n/2})$ for $n=2,3$.
It is important to notice, as pointed out in \cite[Remark 3.1]{Seidman:12}, that these results leave a ``regularity gap''. In other words, the results of \cite{Scott:73} require a $\mathcal{C}^\infty$ domain yet the triangulation is assumed to consist of simplices.
Using a different technique, Casas \cite{Casas:85} obtained the same result for polygonal or polyhedral domains and general regular Borel measures on the right-hand side.
Estimates in other norms are also available in the literature \cite{MR777267,MR0431753}.

In the context of weighted Sobolev spaces, interpolation estimates
and a priori error estimates have been developed in \cite{AGM,DAngelo:SINUM}
for such problems. We now show how to apply our
polynomial interpolation theory to obtain similar results.

Let $\Omega$ be a convex polyhedral domain in $\R^n$ with Lipschitz boundary, and $x_0$ be an interior point of
$\Omega$. Consider the following elliptic boundary value problem:
\begin{equation}
\label{-lap=delta}
  \begin{dcases}
    -\nabla \cdot (\mathcal{A} \nabla u) + \textbf{b}\cdot  \nabla u + cu  = \delta_{x_0},
    & \text{in } \Omega, \\
    u = 0, & \text{on } \partial\Omega,
  \end{dcases}
\end{equation}
where $\mathcal{A} \in L^\infty(\Omega)$ is a piecewise smooth and uniformly symmetric
positive definite matrix, $\textbf{b} \in W^{1,\infty}(\Omega)^n$, $c \in L^{\infty}(\Omega)$,
and $\delta_{x_0}$ denotes the Dirac delta supported at $x_0 \in \Omega$.
Existence and uniqueness of $u$ in weighted Sobolev spaces follows from 
\cite[Theorem 1.3]{AGM} and Lemma \ref{lem:infsup} below,
 and its asymptotic behavior near $x_0$ is dictated by that of the Laplacian
\begin{equation}\label{asympt-x0}
\nabla u(x) \approx |x-x_0|^{1-n}.
\end{equation}
Denote by $d = \diam(\Omega)$ the diameter of $\Omega$ and by
$\dist(x)$ the scaled Euclidean distance
$\dist(x)=|x-x_0|/(2d)$ to $x_0$. Define the weight
\begin{equation}
\label{eq:defofvarpi}
  \varpi(x) = \begin{dcases}
                \frac{ \dist(x)^{n-2}} { \log^2 \dist(x) }, & 0 < \dist(x) < \frac12, \\
                \frac{2^{2-n}}{\log^2 2}, & \dist(x) \geq \frac12.
              \end{dcases}
\end{equation}
We now study two important properties of $\varpi$: $\GRAD u\in L^2(\varpi,\Omega)$
and $\varpi \in A_2(\R^n)$.
\begin{lemma}[regularity of $\nabla u$]\label{L:nabla-u}
The solution $u$ of \eqref{-lap=delta} satisfies
$\GRAD u \in L^2(\varpi,\Omega)$.
\end{lemma}
\begin{proof}
Since $\Omega\subset B$, the ball of radius $d$ centered at $x_0$, we
readily have from \eqref{asympt-x0}
\begin{equation*}
\int_\Omega |\nabla u|^2 \varpi \lesssim \int_B \dist(x)^{2(1-n)}
\frac{\dist(x)^{n-2}}{\log^2 \dist(x)} \diff x
\lesssim \int_0^{\frac12} \frac{1}{r\log^2 r} \diff r = \frac{1}{\log 2},
\end{equation*}
which is the asserted result.
\end{proof}
\begin{lemma}[$\varpi \in A_2(\R^n)$]\label{L:varpi-A2}
The weight $\varpi$ belongs to the Muckenhoupt class $A_2(\R^n)$ with
constant $C_{2,\varpi}$ only depending on $d$.
\end{lemma}
\begin{proof}
Let $x_0=0$ for simplicity,
let $B_r=B_r(y)$ be a ball in $\R^n$ of radius $r$ and center $y$, and denote
$\varpi(B_r) = \int_{B_r} \varpi$ and
$\varpi^{-1}(B_r) = \int_{B_r} \varpi^{-1}$. We must
show
\begin{equation}\label{muck}
\varpi(B_r) \, \varpi^{-1}(B_r) \lesssim r^{2n} \qquad \forall \, r>0,
\end{equation}
with a hidden constant depending solely on $d$. We split the proof
into two cases.

\begin{enumerate}[1.]
  \item Case $|y|<2r$: Since $B_r(y) \subset B_{3r}(0)$ we infer that
\[
\varpi(B_r) \lesssim \int_{B_{3r}(0)} \frac{\big(\frac{|x|}{2d}\big)^{n-2}}{\log^2 \frac{|x|}{2d}}
\diff x \lesssim \int_0^{\frac{3r}{2d}} \frac{s^{2n-3}}{\log^2 s} \diff s
\approx \frac{\big(\frac{3r}{2d}\big)^{2n-2}}{\log^2\frac{3r}{2d}}
\]
and
\[
\varpi^{-1}(B_r) \lesssim \int_{B_{3r}(0)} \Big(\frac{|x|}{2d}\Big)^{2-n}
\log^2 \Big(\frac{|x|}{2d} \Big) \diff x
\lesssim \int_0^{\frac{3r}{2d}} s \log^2 s \diff s \approx
  \Big(\frac{3r}{2d}\Big)^2 \log^2 \frac{3r}{2d},
\]
provided $3r<d$.
The equivalences $\approx$ can be checked via L'H\^opital's rule
for $r\to0$. If $3r\ge d$, then both $\varpi(B_r)$ and $\varpi^{-1}(B_r)$ are
bounded by constants depending only on $d$.
Therefore, this yields \eqref{muck}.

  \item Case $|y| \ge 2r$: Since all $x\in B_r(y)$ satisfy
$\frac12 |y| \le |x| \le \frac32 |y|$ we deduce
\[
\varpi \le \min\left\{\frac{\big(\frac{3|y|}{4d}\big)^{n-2}}{\log^2
  \frac{3|y|}{4d}}, \frac{2^{2-n}}{\log^2 2} \right\},
\quad
\varpi^{-1} \le \max\left\{\Big(\frac{|y|}{4d}\Big)^{2-n}
\log^2 \frac{|y|}{4d}, 2^{n-2} \log^2 2 \right\},
\]
whence $\varpi(B_r) \, \varpi^{-1}(B_r)$ satisfies again \eqref{muck}.
\end{enumerate}
This completes the proof.
\end{proof}

The fact that the weight $\varpi \in A_2(\R^n)$ is the key property for the analysis of discretizations 
of problem \eqref{-lap=delta}. Let us apply the results of Theorem~\ref{thm:embeddingdiffmetricsI} to this 
particular weight.

\begin{lemma}[$H^1(\Omega) \hookrightarrow L^2(\varpi^{-1},\Omega)$]
\label{lem:h1l2weight}
Let $\varpi$ be defined in \eqref{eq:defofvarpi}. If $n < 4$, then the following embedding holds:
\[
  H^1(\Omega) \hookrightarrow L^2(\varpi^{-1},\Omega).
\]
\end{lemma}
\begin{proof}
This is an application of Theorem~\ref{thm:embeddingdiffmetricsI}. We must show when condition \eqref{weight_scaling} holds with
$p=q=2$, $\omega=1$ and $\rho=\varpi^{-1}$. In other words, we need to verify
\[
  \Lambda(r,R) := \frac{r^{2-n}}{R^{2-n}} \frac{ \varpi^{-1}(B_r) }{ \varpi^{-1}(B_R) } \lesssim 1, \quad \forall r \in (0, R],
\]
where both $B_r$ and $B_R$ are centered at $y\in\R^n$. We proceed as in Lemma~\ref{L:varpi-A2} and consider now three cases.
\begin{enumerate}[1.]
  \item $|y|<2r$. We know from Lemma~\ref{L:varpi-A2} that $\varpi^{-1}(B_r) \lesssim \big(\frac{3r}{2d}\big)^2 \log^2 \big(\frac{3r}{2d}\big)$.
  Moreover, every $x\in B_R(y)$ satisfies $|x| < |y|+R \le 3R$ whence
  \[
    \varpi^{-1}(B_R) \ge  \int_{B_R} \Big(\frac{3|x|}{2d}\Big)^{2-n}
    \log^2 \Big(\frac{3|x|}{2d}\Big) \diff x \approx \int_0^{\frac{3R}{2d}} s \log^2s \diff s
    \approx \Big(\frac{3R}{2d}\Big)^2
    \log^2 \Big(\frac{3R}{2d}\Big).
  \]
  If $n<4$, then this shows
  \[
    \Lambda(r,R)
    \lesssim  \frac{r^{4-n}}{R^{4-n}} \frac{\log^2\big(\frac{3r}{2d}\big)}
        {\log^2\big(\frac{3R}{2d}\big)} \lesssim 1.
  \]

  \item $2r\le |y| <2R$. We learn from Lemma~\ref{L:varpi-A2} that
  \[
    \varpi^{-1}(B_r) \lesssim
     |B_r|\Big(\frac{|y|}{4d}\Big)^{2-n}\log^2\Big(\frac{|y|}{4d}\Big)
    \lesssim \Big(\frac{r}{2d} \Big)^2\log^2 \Big(\frac{r}{2d}\Big).
  \]
  In addition, any $x\in B_R$ satisfies $|x|\le |y|+R\le 3R$ and the same bound as in Case 1 holds for $\varpi^{-1}(B_R)$. Consequently, $\Lambda(r,R)\lesssim 1$ again for $n<4$.
  
  \item $|y|\ge 2R$. Since still $|y|>2r$ we have for $\varpi^{-1}(B_r)$ the same upper bound as in Case 2. On the other hand, for all $x\in B_R$ we realize that $|x| \le |y| + R \le \frac32 |y|$ and $\varpi^{-1}(x)\ge \varpi^{-1}(\frac32 y)$. Therefore, we deduce
  \[
    \Big(\frac{3R}{d}\Big)^2 \log^2 \frac{3R}{d}
    \lesssim R^n \Big(\frac{3|y|}{2d}\Big)^{2-n} \log^2 \Big(\frac{3|y|}{2d}\Big)
    \lesssim \varpi^{-1}(B_R),
  \]
  which again leads to $\Lambda(r,R) \lesssim 1$ for $n<4$.
\end{enumerate}
This concludes the proof.
\end{proof}

The embedding of Lemma~\ref{lem:h1l2weight} allows us to develop a general theory for equations of the form \eqref{-lap=delta} on weighted spaces. To achieve this, define
\begin{equation}
\label{eq:defofforma}
  a(w,v) = \int_\Omega \mathcal{A} \nabla w\cdot \nabla v + \textbf{b}\cdot\nabla w v + cwv.
\end{equation}
The following results follow \cite{AGM,DAngelo:SINUM}.

\begin{lemma}[inf--sup conditions]
\label{lem:infsup}
The bilinear form $a$, defined in \eqref{eq:defofforma}, satisfies
\begin{align}
\label{inf-sup1}
1 &\lesssim \inf_{w \in H^1_0(\varpi,\Omega)}\sup_{v \in H^1_0(\varpi^{-1},\Omega)}
    \frac{a(w,v)}{\|\nabla w\|_{ L^2(\varpi,\Omega) } \|\nabla v\|_{ L^2(\varpi^{-1},\Omega)} },
\\
\label{inf-sup2}
  1 & \lesssim \inf_{v \in H^1_0(\varpi^{-1},\Omega)}\sup_{w \in H^1_0(\varpi,\Omega)}
    \frac{a(w,v)}{\|\nabla w\|_{ L^2(\varpi,\Omega) }\|\nabla v\|_{ L^2(\varpi^{-1},\Omega)} }.
\end{align}
\end{lemma}
\begin{proof}
We divide the proof into several steps:
\begin{enumerate}[1.]
  \item We first obtain an orthogonal decomposition of $L^2(\varpi^{-1},\Omega)$ \cite[Lemma 2.1]{DAngelo:SINUM}:
   for every $\mathbf{q} \in L^2(\varpi^{-1},\Omega)$ there is a unique couple 
  $(\boldsymbol{\sigma},v)\in \polX:=L^2(\varpi^{-1},\Omega) \times H^1_0(\varpi^{-1},\Omega)$ such that
  \begin{gather}
  \label{decomposition}
    \mathbf{q} = \boldsymbol{\sigma} + \nabla v, \qquad
    \int_\Omega \mathcal{A}\boldsymbol{\sigma} \cdot \nabla w = 0, \qquad \forall w \in H^1_0(\varpi,\Omega),
\\
  \label{estimate}
    \|\boldsymbol{\sigma}\|_{L^2(\varpi^{-1},\Omega)} + \|\GRAD v\|_{L^2(\varpi^{-1},\Omega)} 
    \lesssim \|\mathbf{q}\|_{L^2(\varpi^{-1},\Omega)}.
  \end{gather}
  To see this, we let 
    $\polY :=L^2(\varpi^{-1},\Omega) \times H^1_0(\varpi,\Omega)$,
    write \eqref{decomposition} in mixed form 
\begin{equation*}
\mathcal{B}[(\boldsymbol{\sigma},v), (\boldsymbol{\tau},w)] :=
\int_\Omega \boldsymbol{\sigma}\cdot\boldsymbol{\tau}
+ \int_\Omega \nabla v\cdot\boldsymbol{\tau}
+ \int_\Omega \mathcal{A} \boldsymbol{\sigma}\cdot\nabla w
= \int_\Omega \mathbf{q}\cdot\boldsymbol{\tau}
\quad\forall \, (\boldsymbol{\tau},w)\in \polY,
\end{equation*}
and apply the generalized Babu{\v{s}}ka-Brezzi inf--sup theory
\cite[Theorem 2.1]{BCM:88}, \cite[Lemma 2.1]{DAngelo:SINUM}. 
This requires only that $\mathcal{A}$ be positive definite along with
the trivial fact that $\phi \in L^2(\varpi^{-1},\Omega)$
implies $\varpi^{-1} \phi \in L^2(\varpi,\Omega)$.

  \item Set $|\mathbf{b}|=c=0$ and let $w \in H^1_0(\varpi,\Omega)$ be given. According to Step 1 we can 
  decompose $\mathbf{q} = \varpi \nabla w \in L^2(\varpi^{-1},\Omega) $ into 
  $\mathbf{q} = \boldsymbol{\sigma} + \nabla v$. 
Invoking \eqref{decomposition}, as in \cite[Corollary 2.2]{DAngelo:SINUM} 
  and \cite[Proposition 1.1]{AGM}, we infer that
  \[
  \int_{\Omega} \mathcal{A} \nabla w \cdot \nabla v = 
    \int_{\Omega} \mathcal{A}  \nabla w \cdot \mathbf{q}
   - \int_{\Omega}  \mathcal{A}  \nabla w \cdot \mathbf{\sigma} 
   = \int_\Omega \varpi \mathcal{A}  \nabla w \cdot \nabla w
   \approx \int_{\Omega} \varpi | \nabla w |^2, 
  \]
 whence, using \eqref{estimate} in the form $\| \nabla v\|_{L^2(\varpi^{-1},\Omega)} \lesssim \| \nabla w \|_{L^2(\varpi,\Omega)}$, we deduce the inf--sup condition \eqref{inf-sup1}.
  
  \item As in \cite{AGM}, we show that for every $F \in
    H^1_0(\varpi^{-1},\Omega)'$ the problem
  \[
  w \in H^1_0(\varpi,\Omega): \quad
    a(w,v) = \langle F, v \rangle, \quad \forall v \in H^1_0(\varpi^{-1},\Omega),
  \]
  is well posed. To this end, we decompose $w=w_1 + w_2\in
  H^1_0(\varpi,\Omega)$, with
  \begin{gather}
  \label{eq:w1}
  w_1 \in  H^1_0(\varpi,\Omega): \quad
    \int_\Omega \mathcal{A} \nabla w_1 \cdot \nabla v = \langle F, v \rangle, \quad \forall v \in H^1_0(\varpi^{-1},\Omega),
\\
  \label{eq:w2}
  w_2 \in H^1_0(\Omega): \quad
    a(w_2,v) = - \int_\Omega \left( \mathbf{b}\cdot \GRAD w_1 + c w_1 \right) v, \quad \forall v \in H^1_0(\Omega).
  \end{gather}
  In fact, if problems \eqref{eq:w1} and \eqref{eq:w2} have a
  unique solution, then we obtain
  \begin{align*}
    a(w,v) &= a(w_1+w_2,v) \\
    &= \int_\Omega \mathcal{A} \GRAD w_1 \cdot \GRAD v + \int_\Omega \left( \mathbf{b}\cdot \GRAD w_1 + c w_1 \right) v + a(w_2,v) 
    = \langle F, v \rangle,
  \end{align*}
 for any $v \in H^1_0(\varpi^{-1},\Omega) \subset H_0^1(\Omega)$.
  The conclusion of Step 2 shows that \eqref{eq:w1} is well posed. The Cauchy-Schwarz inequality and Lemma~\ref{lem:h1l2weight} yield
  \[
    \int_\Omega \left( \mathbf{b}\cdot \GRAD w_1 + c w_1 \right) v \lesssim
    \| w_1 \|_{H^1(\varpi,\Omega)}  \| v \|_{L^2(\varpi^{-1},\Omega)}
    \lesssim \| F\|_{H^1_0(\varpi^{-1},\Omega)'}\| \nabla v\|_{L^2(\varpi^{-1},\Omega)},
  \]
  which combines with the fact that $a(\cdot,\cdot)$ satisfies
   the inf--sup condition in $H^1_0(\Omega)$ \cite[Theorem 5.3.2 - Part
   I]{BA:72} to
  show that \eqref{eq:w2} is well posed as well.
\end{enumerate}
Finally, the general inf--sup theory \cite{MR2050138} \cite[Theorem 2]{NSV} 
guarantees the validity
of the two inf--sup conditions \eqref{inf-sup1} and \eqref{inf-sup2}.
This concludes the proof.
\end{proof}

We also have the following discrete counterpart of Lemma \ref{lem:infsup}.
We refer to \cite[Lemma 3.3]{DAngelo:SINUM} and \cite[Theorem 2.1]{AGM} 
for similar results which, however, do not exploit the Muckenhoupt
structure of the weight $\varpi$.

\begin{lemma}[discrete inf--sup conditions]
\label{lem:discinfsup}
Let $\T$ be a quasi-uniform mesh of size $h$ consisting of simplices. 
If $\V(\T)$ is made of piecewise linears, then the bilinear form $a$, 
defined in \eqref{eq:defofforma}, satisfies:
\begin{align*}
  1 &\lesssim \inf_{W \in \V(\T)}\sup_{V \in \V(\T)}
    \frac{a(W,V)}{\|\nabla W\|_{L^2(\varpi,\Omega)}\|\nabla V\|_{L^2(\varpi^{-1},\Omega)}}, \\
  1 &\lesssim \inf_{V \in \V(\T)}\sup_{W \in \V(\T)}
    \frac{a(W,V)}{\|\nabla W\|_{L^2(\varpi,\Omega)}\|\nabla V\|_{L^2(\varpi^{-1},\Omega)}}.
\end{align*}
where the hidden constants depend on $C_{2,\varpi}$ but not on $h$.
\end{lemma}
\begin{proof}
We proceed as in Lemma~\ref{lem:infsup}. We define the spaces of piecewise constants
\[
  \V_0(\T) = \W_{0}(\T) =
  \left\{ \mathbf{Q} \in L^\infty(\Omega): \mathbf{Q}_{|T} \in \R^n, \ \forall T \in \T \right\},
\]
those of piecewise linears $\V_{1}(\T)=\W_{1}(\T)=\V(\T)$, and endow the product spaces
$\V_{0}(\T)\times\V_{1}(\T)$ and $\W_{0}(\T)\times\W_{1}(\T)$ with the norms
of $\polX$ and $\polY$ respectively, the latter spaces being defined in Lemma~\ref{lem:infsup}.
Given $\mathbf{Q}\in\V_{0}(\T)$, we need the following orthogonal
decomposition --- a discrete counterpart of \eqref{decomposition}-\eqref{estimate}:
find $\Sigma\in\V_0(\T), V\in\V_1(\T)$ so that
\begin{gather}
  \label{decomposition-disc}
    \mathbf{Q} = \boldsymbol{\Sigma} + \nabla V, \qquad
    \int_\Omega \mathcal{A}\boldsymbol{\Sigma} \cdot \nabla W = 0, 
    \qquad \forall W \in \W_{1}(\T),
\\
  \label{estimate-disc}
    \|\boldsymbol{\Sigma}\|_{L^2(\varpi^{-1},\Omega)} + \|\GRAD V\|_{L^2(\varpi^{-1},\Omega)} 
    \lesssim \|\mathbf{Q}\|_{L^2(\varpi^{-1},\Omega)}.
\end{gather}

We first have to verify that the bilinear form $\mathcal{B}$ satisfies
a discrete inf--sup condition, as in Step 1 of Lemma~\ref{lem:infsup}.
We just prove the most problematic inf--sup
\begin{equation*}
\|\nabla W\|_{L^2(\varpi,\Omega)} \lesssim \sup_{\mathbf{T}\in\V_{0}(\T)} \frac{\int_\Omega
\mathcal{A}\mathbf{T}\cdot\nabla W}{\|\mathbf{T}\|_{ L^2(\varpi^{-1},\Omega) }}.
\end{equation*}
We let
$\mathbf{T}=\varpi_\T\nabla W \in \V_{0}(\T)$, where $\varpi_\T$ is
the piecewise constant weight defined on each element $T\in\T$ as
$\varpi_\T|_T = |T|^{-1}\int_T \varpi$.
Since $\nabla W \in \V_{0}(\T)$, we get
\[
\int_\Omega \mathcal{A}\mathbf{T}\cdot\nabla W = 
\int_\Omega \varpi_\T \mathcal{A} \nabla W \cdot\nabla W
\approx \int_\Omega \varpi_\T \nabla W \cdot\nabla W
= \int_\Omega \varpi |\nabla W|^2,
\]
and
\[
\int_\Omega \varpi^{-1}|\mathbf{T}|^2 = \sum_{T\in\T}
\int_T |T|^{-2} \varpi^{-1} \left(\int_T \varpi\right)^2 |\nabla W_{|T}|^2
\le C_{2,\varpi} \int_\Omega \varpi |\nabla W|^2.
\]
We employ a similar calculation to perform Step 2 of
Lemma~\ref{lem:infsup}, and the rest is exactly the same as in 
Lemma~\ref{lem:infsup}. The proof is thus complete.
\end{proof}

The numerical analysis of a finite element approximation to the solution of problem \eqref{-lap=delta} is now a consequence of the interpolation estimates developed in section~\ref{sub:sec:diff_weight}.

\begin{corollary}[error estimate for elliptic problems with Dirac sources]
\label{co:delta_estimates}
Assume \\
that $n<4$ and let $u \in H_0^1(\varpi,\Omega)$ be the solution of \eqref{-lap=delta} and
$U_{\T} \in \V(\T)$ be the finite element solution to
\eqref{-lap=delta}. If $\T$ is simplicial, quasi-uniform
and of size $h$, we have the following error estimate
\begin{equation}
\label{estimate_co_2}
 \| u - U_{\T} \|_{L^2(\Omega)} \lesssim h^{2-n/2}|\log h| \| \nabla u \|_{L^2(\varpi,\Omega)}.
\end{equation}
\end{corollary}
\begin{proof}
We employ a duality argument. Let $\varphi \in H_0^1(\Omega)$ be the solution of
\begin{equation}
\label{-lap=u-U}
a(v,\varphi) =
\int_\Omega (u-U_\T) v \quad\forall \, v\in H^1_0(\Omega),
\end{equation}
which is the adjoint of \eqref{-lap=delta}.
Since $\Omega$ is convex and polyhedral, and the coefficients
$\mathcal{A},\textbf{b},c$ are sufficiently smooth, we have the
standard regularity pick-up \cite{GT}:
\begin{equation}
\label{phiH^2}
\| \varphi \|_{H^2(\Omega)} \lesssim \| u - U_{\T}\|_{L^2(\Omega)}.
\end{equation}
This, together with Lemma~\ref{lem:h1l2weight}, allows us to conclude that, if $n<4$,
\[
  \varphi \in H^2(\Omega) \cap H^1_0(\Omega) \hookrightarrow H^1_0(\varpi^{-1},\Omega).
\]
Moreover, Theorem~\ref{thm:W1qrhow2pomega} yields the error estimate
\[
  \|\nabla(\varphi-\Pi_\T\varphi)\|_{L^2(\varpi^{-1},\Omega)} \lesssim \sigma(h) \|\varphi\|_{H^2(\Omega)}.
\]
with
\[
  \sigma(h) = h \big(\varpi^{-1}(B_h)\big)^{\frac12} |B_h|^{-\frac12} \lesssim h^{2-\frac{n}{2}} |\log h|.
\]

Let $\Phi_{\T}\in\V(\T)$ be the Galerkin solution to
\eqref{-lap=u-U}. Galerkin orthogonality and the continuity of the form $a$ on
$H^1_0(\varpi,\Omega)\times H^1_0(\varpi^{-1},\Omega)$ yield
\begin{equation}
 \| u - U_{\T}\|^2_{L^2(\Omega)}
 = a( u , \varphi - \Phi_{\T} )
  \lesssim \| \nabla u \|_{L^2(\varpi,\Omega)}
\| \nabla ( \varphi - \Phi_{\T} ) \|_{L^2(\varpi^{-1},\Omega)}.
\label{eq:L2est}
\end{equation}
The discrete inf--sup conditions of Lemma~\ref{lem:discinfsup} and and the continuity of the form $a$
allow us to conclude that
\[
  \| \nabla ( \varphi - \Phi_{\T} ) \|_{L^2(\varpi^{-1},\Omega)} \lesssim \| \nabla ( \varphi - \Pi_{\T}\varphi ) \|_{L^2(\varpi^{-1},\Omega)}.
\]
Combining this bound with \eqref{phiH^2} and \eqref{eq:L2est} results in
\[
  \| u - U_{\T}\|^2_{L^2(\Omega)} \lesssim \sigma(h)
  \| \nabla u \|_{L^2(\varpi,\Omega)} \| u - U_{\T}\|_{L^2(\Omega)},
\]
which is the asserted estimate \eqref{estimate_co_2} in disguise.
\end{proof}

\begin{remark}[an interpolation result]
For any $\beta \in (-n,n)$ we can
consider the weight $\dist(x)^\beta$, which belongs to the $A_2(\R^n)$
Muckenhoupt class. Theorem~\ref{TH:v-Pikp}
and Theorems \ref{full_aniso_W1p} and \ref{TH:v - PivH1boundary} show that
\[
  \| u - \Pi_\T u \|_{L^2(\distb \,\, ,\Omega)} \lesssim \| h \nabla u \|_{L^2(\distb \,\, ,\Omega)}.
\]
This extends the interpolation error estimates of \cite[Proposition 4.6]{AGM}, which
are valid for $\beta \in (-n,0)$ only.
\end{remark}

\subsection{Fractional powers of uniformly elliptic operators}
\label{sub:fracL}

We finally examine finite element approximations
of solutions to fractional differential equations; we refer the
reader to \cite{NOS}
for further details. Let $\Omega$ be a polyhedral domain in $\R^n$ ($n\ge1$), with boundary
$\partial\Omega$. Given a piecewise smooth and uniformly symmetric positive definite matrix
$\mathcal{A} \in L^\infty(\Omega)$ and a nonnegative function $c \in L^\infty(\Omega)$, define
the differential operator
\[
  \calL w = - \DIV( \mathcal{A} \nabla w ) + c w.
\]
Given $f \in H^{-1}(\Omega)$, the problem of finding
$u \in H_0^1(\Omega)$ such that $\calL u = f$ has a unique solution.
Moreover, the operator $\calL : \mathcal{D}(\calL) \subset L^2(\Omega) \rightarrow L^2(\Omega)$ with domain
$\mathcal{D}(\calL) = H^2(\Omega) \cap H^1_0(\Omega)$ has a compact inverse
\cite[Theorem 2.4.2.6]{Grisvard}.
Therefore, there exists a sequence of eigenpairs
$\{\lambda_k,\varphi_k\}_{k=1}^\infty$, with $\lambda_k > 0$, such that
\[
  \calL \varphi_k = \lambda_k \varphi_k, \ \text{in } \Omega \qquad {\varphi_k}_{|\partial\Omega} = 0.
\]
The sequence $\{\varphi_k\}_{k=1}^\infty$ is an orthonormal basis of $L^2(\Omega)$.

In this case, for $s\in(0,1)$, we define the fractional powers of $\calL_0$ (where the sub-index is
used to indicate the homogeneous Dirichlet boundary conditions) by
\[
  w = \sum_k w_k \varphi_k
\quad\Longrightarrow\quad
 \calL_0^s w = \sum_k \lambda_k^s w_k \varphi_k.
\]
It is possible also to show that $\calL_0^s : \Hs \rightarrow \Hsd$
is an isomorphism, where
\begin{equation}
  \label{H}
  \Hs =
  \begin{dcases}
    H^s(\Omega),  & s \in (0,\sr), \\
    H_{00}^{1/2}(\Omega), & s = \sr, \\
    H_0^s(\Omega), & s \in (\sr,1),
  \end{dcases}
\end{equation}
and $\Hsd$ denotes its dual space.
We are interested in finding numerical solutions to the following fractional differential equation:
given $s\in (0,1)$ and a function $f \in \Hsd$,
find $u$ such that
\begin{equation}
\label{fl=f_bdddom}
    \calL_0^s u = f.
\end{equation}

The fractional operator $\calL_0^s$ is {\it nonlocal} (see \cite{Landkof,CS:07,CS:11}).
To localize it, Caffarelli and Silvestre showed in \cite{CS:07} that any power of the fractional Laplacian
in $\R^n$ can be determined as a Dirichlet-to-Neumann operator via an extension problem on
the upper half-space $\R^{n+1}_+$. For a bounded domain $\Omega$ and more general operators, this result
has been extended and adapted in \cite{CDDS:11} and \cite{ST:10}, respectively.
This way the nonlocal problem \eqref{fl=f_bdddom} is replaced by
the local one
\[
-\DIV\big(y^\alpha \mathbf{A} \nabla \ue) + y^\alpha c\ue = 0
\]
with $\alpha := 1-2s$,
$ \mathbf{A} = \diag\{\mathcal A, 1\} \in \R^{(n+1) \times (n+1)}$,
posed in the semi-infinite cylinder
\[
  \C = \left\{ (x',y) : x' \in \Omega, \ y \in (0,\infty) \right\},
\]
and subject to a Neumann condition at $y=0$ involving $f$.
Since $\C$ is an unbounded domain, this problem cannot be directly approximated with
finite-element-like techniques. However, as \cite[Proposition 3.1]{NOS} shows, the solution to this problem decays exponentially
in the extended variable $y$ so that, by truncating the cylinder $\C$ to
\[
  \C_\Y = \Omega \times (0,\Y),
\]
and setting a vanishing Dirichlet condition on the upper boundary
$y = \Y$, we only incur in an
exponentially small error in terms of $\Y$ \cite[Theorem 3.5]{NOS}.

Define
\[
  \HL(y^{\alpha},\C_\Y) = \left\{ v \in H^1(y^\alpha,\C_\Y): v = 0 \text{ on }
    \partial_L \C_\Y \cup \Omega \times \{ \Y\} \right\},
\]
where $\partial_L \C_\Y = \partial\Omega \times (0,\Y)$ is the lateral boundary.
As \cite[Proposition 2.5]{NOS} shows,
the trace operator $\HL(y^{\alpha},\C_\Y) \ni w \mapsto \tr w \in \Hs$ is well defined. The
aforementioned problem then reads: find $\ue \in \HL(y^{\alpha}, \C_\Y)$ such that
for all $v \in \HL(y^{\alpha},\C_\Y)$
\begin{equation}
\label{alpha_harmonic_extension_weak_T}
  \int_{\C_\Y} y^\alpha\left( (\mathbf{A} \nabla \ue) \cdot \nabla v
  + c \ue v \right)
  = d_s \langle f, \tr v \rangle_{\Hsd \times \Hs },
\end{equation}
where $\langle \cdot, \cdot \rangle_{\Hs \times \Hsd}$
denotes the duality pairing between $\Hs$ and
$\Hsd$ and $d_s$ is a positive normalization constant
which depends only on $s$.

The second order regularity of the solution $\ue$ of \eqref{alpha_harmonic_extension_weak_T},
with $\C_{\Y}$ being replaced by $\C$, is much worse in the pure $y$ direction as the following estimates
from \cite[Theorem 2.6]{NOS} reveal
\begin{align}
\label{reginx}
  \| \calL _{x'} \ue\|_{L^2(y^\alpha,\C)} + \| \partial_y \nabla_{x'} \ue \|_{L^2(y^\alpha,\C)} & \lesssim \| f \|_{\Ws}, \\
\label{reginy}
  \| \ue_{yy} \|_{L^2(y^\beta,\C)} &\lesssim \| f \|_{L^2(\Omega)},
\end{align}
where $\beta > 2\alpha + 1$. This suggests that
{\it graded} meshes in the extended variable $y$ play a fundamental role.

We construct a mesh over $\C_\Y$ with cells of the form
$T=K\times I$ with $K \subset \Omega$ being an element that is
isoparametrically equivalent either to $[0,1]^n$ or the unit
  simplex in $\R^n$
and $I \subset \R$ is an interval. Exploiting the Cartesian structure of the mesh it is possible to handle
anisotropy in the extended variable and, much as in \S~\ref{subsec:intQ}, obtain estimates of the form
\begin{align*}
  \| v - \Pi_{\T} v \|_{L^2(y^\alpha,T)} & \lesssim h_{\vero'}  \| \nabla_{x'} v\|_{L^2(y^\alpha,S_T)} +
  h_{\vero''}\| \partial_y v\|_{L^2(y^\alpha,S_T)}, \\
  \| \partial_{x_j}(v - \Pi_{\T} v) \|_{L^2(y^\alpha,T)} &\lesssim
 h_{\vero'}  \| \nabla_{x'} \partial_{x_j} v\|_{L^2(y^\alpha,S_T)} +
  h_{\vero''}\| \partial_y \partial_{x_j} v\|_{L^2(y^\alpha,S_T)},
\end{align*}
with $j=1,\ldots,n+1$ and where $h_{\vero'}= \min\{h_K: \vero' \textrm{ is a vertex of } K\}$, and
$h_{\vero''} = \min\{h_I: \vero'' \textrm{ is a vertex of } I\}$; see \cite[Theorems 4.6--4.9]{NOS} for details.
However, since $\ue_{yy} \approx y^{-\alpha -1 }$ as $y \approx 0$,
we realize that $\ue \notin H^2(y^{\alpha},\C_{\Y})$
and the second estimate is not meaningful for $j=n+1$.
In view of the regularity estimate \eqref{reginy}
it is necessary to measure the regularity of $\ue_{yy}$ with a
stronger weight and thus compensate with a graded mesh in the extended
dimension. This makes anisotropic estimates essential.

We consider the graded partition of the interval $[0,\Y]$ with mesh points
\begin{equation}
\label{graded_mesh}
  y_k = \left( \frac{k}{M}\right)^{\gamma} \Y, \quad k=0,\dots,M,
\end{equation}
where $\gamma > 3/(1-\alpha)$, along with
a quasi-uniform triangulation $\T_{\Omega}$ of the domain $\Omega$.
We construct the mesh $\T_{\Y}$ as the tensor product of $\T_\Omega$
and the partition given in \eqref{graded_mesh};
hence $\#\T = M \, \# \T_\Omega$. Assuming that $\# \T_\Omega \approx M^n$
we have $\#\T_\Y \approx M^{n+1}$. Finally, since $\T_\Omega$ is shape regular and quasi-uniform,
$h_{\T_{\Omega}} \approx (\# \T_{\Omega})^{-1/n}$. All these considerations allow us to obtain the
following result.

\begin{corollary}[error estimate for fractional powers of elliptic operators]
\label{TH:fl_error_estimates}
Let $\T$ be a graded tensor product grid, which is quasi-uniform in
$\Omega$ and graded in the extended variable so that
\eqref{graded_mesh} hold. If $\V(\T)$ is made of bilinear elements, then 
the solution of \eqref{alpha_harmonic_extension_weak_T}
and its Galerkin approximation $U_{\T} \in \V(\T)$ satisfy
\begin{equation*}
  \| \ue - U_{\T} \|_{\HLn(y^\alpha,\C)} \lesssim
|\log(\# \T_{\Y})|^s(\# \T_{\Y})^{-1/(n+1)} \|f \|_{\mathbb{H}^{1-s}(\Omega)},
\end{equation*}
where $\Y \approx \log(\# \T_{\Y})$. Alternatively, if $u$
  denotes the solution of \eqref{fl=f_bdddom}, then
\begin{equation*}
\| u - U_{\T}(\cdot,0) \|_{\Hs} \lesssim
|\log(\# \T_{\Y})|^s(\# \T_{\Y})^{-1/(n+1)} \|f \|_{\mathbb{H}^{1-s}(\Omega)}
\end{equation*}
\end{corollary}
\begin{proof}
First of all, notice that $y^{\alpha} \in A_2(\R^{n+1})$ for $\alpha \in (-1,1)$.
Owing to the exponential decay of $\ue$, and the choice
of the parameter $\Y$, it suffices to estimate $\ue - \Pi_{\T_{\Y}} \ue$ on the mesh $\T_{\Y}$;
see \cite[\S~4.1]{NOS}.
To do so,  we notice that if
$I_1$ and $I_2$ are neighboring cells on the partition of $[0,\Y]$,
then the weak regularity condition \eqref{shape_reg_weak} holds.
Thus, we decompose the mesh $\T_\Y$ into the sets
\begin{align*}
  \Tm_0 = \left\{ T \in \T_\Y: \ S_T \cap (\bar\Omega \times \{0\} )
    = \emptyset \right\}, \quad
  \Tm_1 = \left\{ T \in \T_\Y: \ S_T \cap (\bar\Omega \times \{0\} )
    \neq \emptyset \right\},
\end{align*}
and apply our interpolation theory developed in
Theorems \ref{full_aniso_W1p} and \ref{TH:v - PivH1boundary}
for interior and boundary elements
respectively, together
with the local regularity estimates for the function $\ue$ derived in \cite[Theorem 2.9]{NOS}.
\end{proof}

The error estimates with graded meshes are quasi-optimal in both regularity and order.
Error estimates for quasi-uniform meshes are suboptimal in terms of order \cite[Section 5]{NOS}.
Mesh anisotropy is thus able
to capture the singular behavior of the solution $\ue$
and restore optimal decay rates.

\section*{Acknowledgement}
We dedicate this paper to R.G. Dur\'an, whose work at the intersection of real and numerical analysis has been inspirational to us.

\bibliographystyle{plain}
\bibliography{biblio}

\end{document}